\documentclass[twoside,11pt]{article}

\usepackage{jmlr2e}

\usepackage[utf8]{inputenc} 
\usepackage[T1]{fontenc}    
\usepackage{booktabs}       
\usepackage{amsfonts}       
\usepackage{nicefrac}       
\usepackage{microtype}      
\usepackage{lipsum}
\usepackage{graphicx}

\usepackage{amsmath,bm}

\usepackage{subcaption}
\usepackage{float}

\usepackage{amsfonts}
\usepackage{amssymb,amsmath,latexsym,tensor}
\usepackage{mathrsfs}
\usepackage{enumerate}
\usepackage{savesym}
\usepackage{graphicx}
\usepackage{adjustbox}
\usepackage[pdftex]{color}
\usepackage[font=footnotesize,labelfont=bf]{caption}
\savesymbol{AND}
\usepackage{epstopdf}
\usepackage{algorithm}
\usepackage{algorithmic}
\newcommand{\N}{\mathbb{N}}

\newcommand{\R}{\mathbb{R}}

\DeclareMathOperator{\E}{\mathbb{E}}

\DeclareMathOperator*{\argmin}{arg\,min}
\DeclareMathOperator*{\argmax}{arg\,max}

\usepackage{mathtools}
\DeclarePairedDelimiterX{\norm}[1]{\lVert}{\rVert}{#1}
\DeclarePairedDelimiterX{\inp}[2]{\langle}{\rangle}{#1, #2}

\usepackage{cleveref}

\usepackage{array,multirow}

\ShortHeadings{Stochastic Steepest Descent with Momentum}{Morshed, Sabbir and Noor-E-Alam}
\firstpageno{1}

\begin{document}

\title{Stochastic Steepest Descent Methods for Linear Systems: Greedy Sampling \& Momentum}

\author{\name Md Sarowar Morshed \email morshed.m@northeastern.edu \\
       \addr Department of Mechanical \& Industrial Engineering\\
       Northeastern University\\
       Boston, MA, USA \\
       \AND
       \name Sabbir Ahmad \email ahmad.sab@northeastern.edu \\
       \addr Khoury College of Computer Sciences \\
       Northeastern University\\
       Boston, MA, USA \\
       \AND
       \name Md Noor-E-Alam \email mnalam@neu.edu \\
       \addr Department of Mechanical \& Industrial Engineering\\
       Northeastern University\\
       Boston, MA, USA \\}

\editor{}
\maketitle

\begin{abstract}
Recently proposed adaptive \textit{Sketch \& Project} (SP) methods connect several well-known projection methods such as \textit{Randomized Kaczmarz} (RK), \textit{Randomized Block Kaczmarz} (RBK), \textit{Motzkin Relaxation} (MR), \textit{Randomized Coordinate Descent} (RCD), \textit{Capped Coordinate Descent} (CCD) etc. into one framework for solving linear systems. In this work, we first propose a \textit{Stochastic Steepest Descent} (SSD) framework that connects SP methods with the well-known \textit{Steepest Descent} (SD) method for solving positive-definite linear system of equations. We then introduce two greedy sampling strategies in the SSD framework that allow us to obtain algorithms such as \textit{Sampling Kaczmarz Motzkin} (SKM), \textit{Sampling Block Kaczmarz} (SBK), \textit{Sampling Coordinate Descent} (SCD), etc. In doing so, we generalize the existing sampling rules into one framework and develop an efficient version of SP methods. Furthermore, we incorporated the Polyak momentum technique into the SSD method to accelerate the resulting algorithms. We provide global convergence results for both the SSD method and the momentum induced SSD method. Moreover, we prove $\mathcal{O}(\frac{1}{k})$ convergence rate for the Cesaro average of iterates generated by both methods. By varying parameters in the SSD method, we obtain classical convergence results of the SD method as well as the SP methods as special cases. We design computational experiments to demonstrate the performance of the proposed greedy sampling methods as well as the momentum methods. The proposed greedy methods significantly outperform the existing methods for a wide variety of datasets such as random test instances as well as real-world datasets (LIBSVM, sparse datasets from matrix market collection). Finally, the momentum algorithms designed in this work accelerate the algorithmic performance of the SSD methods.
\end{abstract}


\section{Introduction}
In this paper, we focus on solving the following fundamental problem:
\begin{align}
\label{jml2}
  x^* = \argmin_{x \in \R^{n}} \|x-t\|^2_G \quad \textbf{subject to} \quad  Ax = b.
\end{align}
where $G \in \R^{n \times n}$ is a symmetric positive definite matrix and $ t \in \R^{n}$ is an arbitrary vector. Problem \eqref{jml2} is prevalent and central in a wide range of quantitative areas such as \textit{Numerical Linear Algebra}, \textit{Computer Science}, \textit{Scientific Computing}, \textit{Machine Learning}, \textit{Computer Vision}, \textit{Optimization}, \textit{Signal Processing}, \textit{Financial Engineering}, etc. From a \textit{Machine Learning} perspective, problem \eqref{jml2} arises from a wide range of applications such as Gaussian processes \citep{rasmussen_williams_2008}, \textit{Least-Square Support Vector Machines} \citep{SVM:2007}, graph-based \textit{Semi-Supervised Learning} and \textit{Graph-Laplacian} problems \citep{bengio:2006}, \textit{Gaussian Markov Random Fields} \citep{GaussianMR}, etc. Approximate solutions of \textit{Linear Systems} can be of practical benefit in inexact Newton schemes \citep{kaifeng:2012,wang:2013,jacek:2013} that are gaining lots of traction in the field of large-scale optimization. Throughout the paper, we assume that problem \eqref{jml2} is consistent (there exists $x^*$ such that $Ax^* = b$), dimensions are large and $m \gg n$. In a large-scale setting, solving problem \eqref{jml2} with direct methods are infeasible. For instance, Krylov-type method such as \textit{Conjugate Gradient} \citep{Hestenes:1952} is the state-of-the-art method for solving \eqref{jml2} whenever memory usage/full matrix-vector products aren't of practical concern. However, recent works suggest that Kaczmarz/row-action type methods are favorable in that context as they don't require to store large data matrix $A$, and do not calculate the product of large-scale matrices and vectors \citep{strohmer:2008,lewis:2010}.  


In the advent of big data, projection-based iterative methods are gaining popularity in the research community. Most projection based iterative methods for solving problem \eqref{jml2}, which can be interpreted under one big umbrella of \textit{Sketch \& Project} (SP) methods \citep{gower:2015}. In their work, Gower \textit{et. al} showed that iterative methods such as \textit{Randomized Newton}, \textit{Randomized Kaczmarz}, \textit{Randomized Coordinate Descent}, \textit{Random Gaussian Pursuit} and \textit{Randomized Block Kaczmarz}, etc. can be recovered as special cases of the SP method. The original SP framework was proposed for linear systems, however, recently in has been incorporated in several areas such as \textit{Quasi-Newton} methods \citep{gower:2017,NIPS:2018}, \textit{Matrix Pseudo-inverse} \citep{gower2016linearly}, \textit{Convex Feasibility} \citep{necoara:2019}, \textit{Randomized Subspace Newton} \citep{gower2019rsn}, \textit{Variance Reduction} \citep{sega:2018,Gower2020}, \textit{Newton-Raphson} \citep{yuan2020sketched}, \textit{Linear Feasibility} \citep{morshed:sketching}, etc. Several accelerated/momentum variants of the SP methods have been proposed in various contexts \citep{loizou:2017,NIPS:2018,richtrik2017stochastic}.

The performance of the SP methods depend highly on the selection of random sketching matrix at each iteration. In the broader sense of projection-based iterative methods, some important selection strategies are \textit{Uniform Sampling} \citep{strohmer:2008, lewis:2010,wright:2016}, \textit{Maximum Distance Sampling} \citep{motzkin,nutini:2016}, \textit{Kaczmarz Motzkin Sampling} \citep{haddock:2017, haddock:2019, Morshed2019, morshed2020generalization,morshed:momentum,morshed:sketching}, Capped Sampling \citep{bai:2018,gower2019adaptive}. From the perspective of coordinate descent method, some well-known sampling strategies \footnote{These selection rules are frequently used in the context of convex minimization.} are adaptive selection \citep{nesterov:2012,jaggi:2017,khalil:2018}, Gauss-Southwell rule \citep{nutini:2015,tseng:1990}, selection based on duality gap \citep{Csiba:2015}. In a recent work, Gower \textit{et. al} discussed some of the above mentioned selection rules in the SP methods \citep{gower2019adaptive}.

Now, to show the connection between SP methods with other algorithmic developments, we will discuss some of the classical and modern  Kaczmarz methods for solving problem \eqref{jml2}.Kaczmarz method \citep{kaczmarz:1937} is the oldest and one of the simplest projection based iterative methods for solving linear systems. Although the original Kaczmarz method was deterministic, in recent times, randomized Kaczmarz type methods \citep{strohmer:2008} are gaining a lot of attention from the research community. Another classical method is the so-called \textit{Motzkin Relaxation} (MR) method that selects projection hyperplane vai maximum residual. Interestingly, it has been shown recently that the \textit{Perceptron} algorithm in machine learning \citep{ramdas:2014,ramdas:2016} can be sought as a variant of the MR method. Following the work of Strohmer \textit{et. al}, the research on Kaczmarz type methods found successful for solving a wide range of problems such as linear system, linear feasibility, least square, low-rank matrix recovery, etc \citep[see][]{lewis:2010,needell:2010,eldar:2011,zouzias:2013,lee:2013,NEEDELL:2014,agaskar:2014,ma:2015,NEEDELL:2015, blockneddel:2015,needell:2016, quadratic:2016, wright:2016,needell:2016,haddock:2017, greedbai:2018,haddock:2019, Morshed2019, morshed:2018, razaviyayn:2019, haddock:2019,needell2019block, morshed2020generalization}.

To leverage the unique strength of some of the above algorithms, we propose a comprehensive computing framework, which synthesizes the well-known randomized algorithms such as Kaczmarz, Co-ordinate Descent along with iterative methods such as steepest descent. We develop a stochastic steepest descent algorithm, which is equivalent to the steepest descent interpretation of the so-called SP methods. To the best of our knowledge, this is the first time steepest descent method has been connected to a randomized iterative framework. We also proposed two greedy sketching rules that connect several existing rules into one framework that also generates efficient projection algorithms. We introduced the \textit{Sampling Kaczmarz Motzkin} \citep{haddock:2017,haddock:2019} type selection rule in the \textit{Adaptive Sketch \& Project} method that connects two well known sampling rules, i.e., \textit{Randomized Kaczmarz} and \textit{Motzkin Relaxation} (maximum distance rule). Finally, we developed momentum variant algorithms of the proposed SSD method for solving \textit{Linear Systems}. We achieved a better convergence rate and obtained larger momentum parameter range for the momentum algorithms than the one proved in \citep{loizou:2017}. The proposed greedy methods have superior performance compared to the existing methods. Furthermore, the momentum induced SSD accelerates the basic greedy methods for wide variety of test instances.

\subsection{Notation}
Throughout the paper, we follow standard linear algebra notation. $\mathbb{R}^{m\times n}$ and $\mathbb{R}^{m\times n}_+$ will be used to denote the set of real valued matrices and non-negative matrices, respectively. The feasible region of problem \eqref{jml2} is denoted as $\mathcal{X} = \{x = \argmin \|x-t\|^2_G \ \text{s.t} \ Ax = b\}$. The notation $x^+$ will be used to denote the positive part of any real number, i.e., $x^+ = \max \{x,0\}$. Let $A$ be a matrix. By $a_i^T$, $A_{j}$, $A^{\dagger}$, $\|A\|_F$, $\textbf{Rank}(A)$, $\textbf{Range}(A)$, $\textbf{Null}(A)$, $\lambda_{\min}^+(A)$, $\lambda_{\max}(A)$ we denote the $i^{th}$ row, the $j^{th}$ column, the Moore-Penrose pseudo-inverse, Frobenius norm, rank, range, null, the smallest nonzero eigenvalue, the largest eigenvalue of matrix $A$, respectively. Given any sampling rule $\mathcal{R}$, by which the index $i$ will be chosen, we use the notation $\E[\cdot \ | \ i \sim \mathcal{R} ] = E_i[\cdot]$ to denote the expectation with respect to the sampling rule $\mathcal{R}$. Let, $G \in \R^{n \times n}$ be any positive definite matrix. We define, $\langle x, G x \rangle = x^TGx= \|x\|_G^2 $. For any matrices $A and \ B$, the notation $A \succ B$ defines the positive definiteness of the matrix $A-B$.

\subsection{Outline}
In section \ref{sec:contr}, we first discuss existing randomized methods for solving problem \eqref{jml2}. Then we provide a brief summary of the contributions of our proposed work. In section \ref{sec:tech}, we discuss necessary technical backgrounds about the SP methods in general. In section \ref{sec:algorithms}, we discuss the proposed algorithms and the proposed greedy sketching rules. The convergence results of the proposed methods are provided in section \ref{sec:conv}. In section \ref{sec:special}, we discuss some special cases and their convergence results that can be recovered from our proposed methods. In section \ref{sec:num}, comprehensive numerical experiments are carried out to demonstrate the efficiency of the proposed methods. The paper is concluded in section \ref{sec:conclusions} with remarks about future research. In Appendix, we provide the necessary proofs and external experiment results.

\section{Preliminaries \& Our Contributions}
\label{sec:contr}
In this section, we first discuss some preliminary works that deals with solving problem \eqref{jml2} using the so-called \textit{Sketch \& Project} framework.

\paragraph{Sketch \& Project Methods} \citet{gower2019adaptive} discussed the following generic method \citep[see][]{gower:2015, richtrik2017stochastic} for solving \eqref{jml2}: starting from a random point $x_k$, the SP method updates $x_{k+1}$ as the solution of the following problem:
\begin{align}
\label{eq:jml3}
x_{k+1} =     \argmin_x \|x-x_k\|^2_B \quad \text{subject to} \quad  S_i^TAx = S_i^T b,
\end{align}
where, $S_i \sim \mathcal{D}$. The solution of the above sketched problem has the following closed form:
\begin{align}
\label{eq:jml4}
x_{k+1} = x_k - B^{-1} A^TH_i(Ax_k-b)
\end{align}
where, $H_i = S_i (S_i^TAB^{-1}A^TS_i)^{\dagger}S_i^T$. A general method with $\omega \in (0,2)$ has been proposed by \citet{richtrik2017stochastic}. Furthermore, several variants of SP method have been proposed recently \citep[see][]{loizou:2017, NIPS:2018, gower2019adaptive}.


\paragraph{Sampling Kaczmarz Motzkin (SKM)} \citet{haddock:2017} proposed the following generic method which has been considered for solving linear feasibility:
\begin{align}
\label{jml5}
& x_{k+1}  = x_k - \omega  \frac{\left(a_{i}^Tx_k -b_{i}\right)^+}{\|a_{i}\|^2} a_{i},
\end{align}
where, index $i$ at iteration $k$ is chosen by the following rule: denote $\phi_k(\tau)$ as the collection of $\tau$ rows uniformly sampled from the rows of matrix $A$, then set $i = \argmax_{i \in \phi_k(\tau)} \{a_i^Tx_k-b_i, 0\}$. This has been extended for linear systems by \citet{haddock:2019}. Morshed \textit{et. al} developed several accelerated and momentum variants \citep[see][]{Morshed2019, morshed2020generalization,morshed:momentum,morshed:sketching} of the SKM method for solving linear feasibility problems. It has been shown that, this specific selection rule outperforms the traditional sampling rules such as uniform sampling \citep{strohmer:2008, lewis:2010}, max. distance sampling/Motzkin sampling \citep{motzkin}.

\subsection{Summary of Our Contributions}
\paragraph{Stochastic Steepest Descent Method.} We propose a stochastic steepest descent framework of the so-called SP method by introducing a positive definite matrix $G$. In doing so, we connect several well-known randomized methods such as Kaczmarz, Co-ordinate Descent and exact method such as the Steepest Descent method into one framework.

\begin{table}[h!]
\centering
\caption{SSD methods with momentum for solving problem \eqref{jml2}. See \eqref{jml25} for $\mathcal{W}_k$, $(\lambda_i, u_i)$ is the eigenvalue-eigenvector pair, i.e., $Au_i = \lambda_i u_i$, CD- Co-ordinate Descent, SSD- Stochastic Spectral Descent}
\adjustbox{max width=\textwidth}{
\begin{tabular}{|c|c|c|c|c|}
\hline
$S_i$ & $B$ & Sampling Rule, $q = m$ & $x_{k+1}$ & Kaczmarz \\ \hline
\multirow{2}{*}{$e_i$} & \multirow{2}{*}{$I$} & $i = \argmax_{j \in \phi_k(\tau)} \frac{|a_j^Tx_k-b_j|^2}{\|a_j\|^2} $ & \multirow{2}{*}{\begin{tabular}[c]{@{}c@{}}$x_k + \gamma(x_k-x_{k-1})$\\ $- \omega \frac{a_{i}^Tx_k -b_{i}}{ \|G^{-\frac{1}{2}}a_{i}\|^2} G^{-1} a_{i} $\end{tabular}} & \begin{tabular}[c]{@{}c@{}}Greedy \\ rule \end{tabular} \\ \cline{3-3} \cline{5-5} 
 &  & Capped Sampling ($i \in \mathcal{W}_k$)  &  & \begin{tabular}[c]{@{}c@{}} Capped \\ rule \end{tabular} \\ \hline
$S_i$ & $B$ & Sampling Rule & $x_{k+1}$ & CD-PD \\ \hline
\multirow{2}{*}{$e_i$} & \multirow{2}{*}{$A$} & $i = \argmax_{j \in \phi_k(\tau)} \frac{|a_j^Tx_k-b_j|^2}{A_{jj}} $ & \multirow{2}{*}{\begin{tabular}[c]{@{}c@{}}$x_k + \gamma(x_k-x_{k-1})$\\$- \omega \frac{a_{i}^Tx_k -b_{i}}{\|Ae_i\|^2_{G^{-1} }} G^{-1}A e_{i}$\end{tabular}} & \begin{tabular}[c]{@{}c@{}}Greedy \\ rule \end{tabular} \\ \cline{3-3} \cline{5-5} 
 &  & Capped Sampling ($i \in \mathcal{W}_k$) &  & \begin{tabular}[c]{@{}c@{}} Capped \\ rule \end{tabular} \\ \hline
$S_i$ & $B$ & Sampling Rule & $x_{k+1}$ & CD-LS \\ \hline
\multirow{2}{*}{$Ae_i$} & \multirow{2}{*}{$A^TA$} & $i =  \argmax_{j \in \phi_k(\tau)} \frac{|A_j^T(Ax_k -b)|^2}{\|A_j\|^2} $ & \multirow{2}{*}{\begin{tabular}[c]{@{}c@{}}$x_k + \gamma(x_k-x_{k-1})$\\ $-\omega \frac{A_i^T(Ax_k -b)}{\|A^TAe_i\|^2_{G^{-1} }} G^{-1}A^TA e_{i} $\end{tabular}} & \begin{tabular}[c]{@{}c@{}} Greedy \\ rule \end{tabular} \\ \cline{3-3} \cline{5-5} 
 &  & Capped Sampling ($i \in \mathcal{W}_k$) &  & \begin{tabular}[c]{@{}c@{}}Capped \\ rule \end{tabular} \\ \hline
$S_i$ & $B$ & Sampling Rule, $q = m=n$ & $x_{k+1}$ & SSD \\ \hline
\multirow{2}{*}{$u_i$} & \multirow{2}{*}{$A$} & $i = \argmax_{j \in \phi_k(\tau)} \frac{|\lambda_j u_j^Tx_k-u_j^Tb|^2}{\lambda_j} $  & \multirow{2}{*}{\begin{tabular}[c]{@{}c@{}}$x_k + \gamma(x_k-x_{k-1})$\\$- \omega \frac{\lambda_i u_i^Tx_k-u_i^Tb}{\lambda_i \|G^{-\frac{1}{2}}u_{i}\|^2} G^{-1}u_{i}$\end{tabular}} & \begin{tabular}[c]{@{}c@{}} Greedy \\ rule \end{tabular} \\ \cline{3-3} \cline{5-5} 
 &  & Capped Sampling ($i \in \mathcal{W}_k$) &  & \begin{tabular}[c]{@{}c@{}} Capped \\ rule \end{tabular} \\ \hline
$S_i$ & $B$ & Sampling Rule (exact) & $x_{k+1}$ & SD \\ \hline
\begin{tabular}[c]{@{}c@{}}$A$\end{tabular} & \begin{tabular}[c]{@{}c@{}}$A$\end{tabular} & \begin{tabular}[c]{@{}c@{}}$A \succ 0, \ G = I$\end{tabular} & \begin{tabular}[c]{@{}c@{}}$x_k + \gamma(x_k-x_{k-1})$\\ $-\omega \frac{\|Ax_k-b\|^2}{\|Ax_k-b\|^2_A} (Ax_k-b)$\end{tabular} & \begin{tabular}[c]{@{}c@{}}SD\\ momentum\end{tabular} \\ \hline
\end{tabular}}
\label{tab:AK}
\end{table}

\paragraph{Greedy Sketching Rules.} We extend the scope of the SKM type sampling rules into the more generalized SSD method setting. In doing so, we connect several existing sampling rules such as uniform sampling, max. distance sampling. We also introduce greedy capped rule that extends the so-called capped rule in the proposed SSD framework.

\paragraph{SSD Method with Momentum.} We introduce the well-known heavy ball momentum technique to the developed SSD method. The proposed momentum induced SSD algorithms outperform the basic SSD method on a variety of test instances. Our momentum framework connects several momentum algorithms into one framework. For instance, for the momentum SP method (SSD with $G =B$), we obtain a better convergence result than the obtained in \citep{loizou:2017}. In Table \ref{tab:AK}, we provide special cases of the SSD method along with their respective momentum variants \footnote{Note that in Table \ref{tab:AK}, we present some simplest variants of SSD method. By varying parameters $S_i, \ B, \ G, \ \mathcal{R}$, one can obtain a large array of specialized methods. For instance, with the choice $B = I, \ S_i = I_{:C}$, we get the greedy/momentum version of the randomized block Kaczmarz method proposed by \citet{haddock:2019}. Similarly, with the choice $B = A \succ 0, \ S_i = I_{:C}$, we get a greedy/momentum version of the so-called randomized coordinate Newton descent proposed by \citet{qu:2016}. $I_{:C}$ denotes the column sub-matrix of the $m \times m$ identity matrix indexed by a random set $C$. }.

\paragraph{Global Linear Rate.}
We provide convergence results for the proposed SSD method as well as the momentum SSD method. We establish global convergence results for a wide range of projection parameters $0 < \omega < 2$ and momentum parameter $\gamma \geq 0$. From our convergence results, one can recover convergence results of well-known methods such as steepest descent, Kaczmarz method, Motzkin method, Co-ordinate descent method etc. For the momentum SP method, we obtain a better rate than the existing rate.

\paragraph{Sub-linear Rate.}
We also show that under mild condition, the Cesaro average of iterates, i.e., $\Tilde{x}_k = 1/k \sum \nolimits_{i =0}^k x_i $ generated by SSD and momentum SSD enjoys $\mathcal{O}(1/k)$ sub-linear convergence rate. From our result, we obtain a Cesaro result for the steepest descent method.

\section{Technical Tools}
\label{sec:tech}

In this section, we discuss the general framework for analyzing SP methods proposed by \citet{richtrik2017stochastic}. We start by providing the required assumptions of this work. Then, we introduce function $f(x)$ that is frequently used in literature to analyze SP methods \citep{richtrik2017stochastic,loizou:2017}. At the end of this section, we briefly discuss the stochastic reformulation method proposed  by \citet{loizou:2017}.

\paragraph{Assumptions.} Throughout the paper, we assume the following: (1) the system $Ax = b$ has a solution, and (2) matrix $A$ has no zero rows. Moreover, we assume the following:
\begin{align}
\label{assumption}
  \text{Either} \quad  \sum \limits_{i =1}^{q} H_i \succ \mathbf{0} \quad \text{or} \quad \textbf{Null}\left(\sum \limits_{i =1}^{q} H_i\right) \subseteq \textbf{Null}\left(A^T \right)
\end{align}
where, $H_i = S_i (S_i^TAB^{-1}A^TS_i)^{\dagger}S_i^T$ and $S_i$ is selected following the greedy sketching rule.

\paragraph{Function $f(x)$.} At iteration $k$, the SP method selects a new random matrix $S_i$ based on the sketching rule $\mathcal{R}$. Let's define function $f(x)$ as follows:
\begin{align}
\label{jml6}
 f(x) = \E[f_i(x) \ | \ i \sim \mathcal{R}], \quad   f_{i}(x)  = \frac{1}{2} \|Ax-b\|^2_{H_i} = \frac{1}{2} \|x-x^*\|^2_{Z_i},
\end{align}
where, $ H_i = S_i (S_i^TAB^{-1}A^TS_i)^{\dagger}S_i^T, \ Z_{i} = A^T H_i A$, $B \succ 0$ and $S_i \in \R^m$ is the $i^{th}$ sketching matrix selected from the set $\mathcal{S}(q) = \{S_1,S_2,...,S_q\}$ based on rule $\mathcal{R}$. Note that, the gradients $\nabla f_i$ and $\nabla^G f_i$ of function $f_i$ are given by
\begin{align}
\label{jml7}
  \nabla f_i(x) = A^T H_i(Ax-b), \quad   \nabla^{G} f_i(x) = G^{-1}  A^T H_i(Ax-b),
\end{align}
where, $\nabla^G f_i$ denotes the gradient of $f_i$ with respect to the $G-$ norm. Furthermore, we have
\begin{align*}
  \nabla f(x) = \E[\nabla f_i(x)]= A^T\E[H](Ax-b) = \E[Z](x-x^*), \  \nabla^2 f= \E[Z]
\end{align*}

\paragraph{Stochastic Reformulation.} The following reformulation of problem \eqref{jml2} has been proposed by \citet{richtrik2017stochastic}:
\begin{align}
    \label{jml8}
    x = \argmin f(x) = \frac{1}{2} \|Ax-b\|^2_{\E[H]} = \frac{1}{2} \|x-x^*\|^2_{\E[Z]}.
\end{align}
The optimal function value is given by $f^* = 0$. It can be easy to check that if $\E[H] \succ 0$ holds, then $\|Ax-b\|_{\E[H]} = 0$ implies $Ax-b = 0$. Similarly, when $Ax-b = 0$, we have $\|Ax-b\|_{\E[H]} = 0$. This implies problem \eqref{jml2} and problem \eqref{jml8} are equivalent when $\E[H] \succ 0$ holds. \citet{richtrik2017stochastic} showed that the equivalency of problem \eqref{jml2} and problem \eqref{jml8} can be proven under weaker conditions (this property is denoted as exactness). They proved that problem \eqref{jml2} and problem \eqref{jml8} are equivalent provided that any of the following conditions holds: $\E[H] \succ \mathbf{0}$ or $\textbf{Null}\ (\E[H]) \subseteq \textbf{Null}\ (A^T)$. 

\section{Algorithms}
\label{sec:algorithms}
In this section, we first propose a \textit{Stochastic Steepest Descent} framework that is equipped with greedy sampling rule $\mathcal{R}$. The proposed method generalizes the methods proposed by \citet{gower:2015}, \citet{richtrik2017stochastic}, and \citet{gower2019adaptive}. Our framework allows one to design efficient algorithms based on greedy sampling strategies. The following simplified expressions will be used throughout the paper.

\begin{definition}
\label{jml17}
For any $i \in \N$, let us define the following:
\begin{align*}
   &  T_{i} = G^{-\frac{1}{2}}Z_{i} G^{-\frac{1}{2}}, \ \mathbf{T} = \sum\nolimits_{j=1}^{q} T_j, \ \mu_2(i) = \lambda_{\max}(T_{i}), \ \ \mu_1(i) = \lambda_{\min}(T_{i}), \ \ \mu^+_1(i) = \lambda^+_{\min}(T_{i}) \\
   & r_k = G^{\frac{1}{2}} (x_k-x^*), \ \ \E[T] = \E[T_{i} \ | \ i \sim \mathcal{R}], \  \mu_2 = \max_{i} \mu_2(i) , \  \mu^+_1 = \min_{i} \mu^+_1(i), \ \ \sigma_i = \frac{ \mu_2(i)}{\mu^+_1(i)}
\end{align*}
\end{definition}

\subsection{Stochastic Steepest Descent (SSD)}
Before, we delve into the SSD method, we discuss the main motivation of this work. In the original \textit{Sketch \& Project} method ($\omega =1, \ Z_{i_k} B^{-1} Z_{i_k} = Z_{i_k}$), we have the following identity:
\begin{align*}
    f_{i_k}(x_{k+1}) = \frac{1}{2} \|(I-B^{-1}Z_{i_k})(x_{k}-x^*)\|^2_{Z_{i_k}} = 0
\end{align*}
This motivates us to search for method that gradually decreases the function value, i.e. $f_{i_k}(x_{k+1}) < f_{i_k}(x_k) $. If we set 
steepest direction $d_k = - \nabla^{G} f_i(x_k)$, then we set $x_{k+1}$ as,
\begin{align}
\label{jml9}
    x_{k+1} = x_k + \bm{\alpha} d_k = x_k - \bm{\alpha} \nabla^{G} f_i(x)  \quad \textbf{s.t} \quad \bm{\alpha} = \argmin_{\alpha} f_i(x_{k+1})
\end{align}
Now, by solving the problem we derive the expression for $\bm{\alpha}$ as follows:
\begin{align}
\label{jml10}
    \bm{\alpha} = \argmin_{\alpha} f_i(x_{k+1}) = \argmin_{\alpha } f_i(x_k-  \bm{\alpha} \nabla^{G} f_{i}(x_k)) 
    = \frac{\|\nabla^{G} f_{i}(x_k)\|^2_G}{\|\nabla^{G} f_{i}(x_k)\|^2_{Z_{i_k}}}
\end{align}
Note that, with this choice, we get
\begin{align}
\label{jml11}
    f_i(x_{k+1}) = f_i(x_k) - \frac{\|\nabla^{G} f_{i}(x_k)\|^4_G}{2\|\nabla^{G} f_{i}(x_k)\|^2_{Z_{i_k}}} < f_i(x_k)
\end{align}
In doing so, we extended the scope of the so-called SP methods to a more general framework. Surprisingly, with the choice $G =I, \ B = A, \ S = A \succ 0$, we get the following:
\begin{align}
\label{steepsest}
   x_{k+1} = x_k - \alpha_k (Ax-b), \ \ \alpha_{k} = \frac{\|Ax_k-b\|^2}{\|Ax_k-b\|^2_A}
\end{align}
which is precisely the steepest descent method for solving linear system $Ax = b, \ A \succ 0$. 

\begin{algorithm}
\caption{SSD Algorithm: $x_{k+1} = \textbf{SSD}(A,b,x_0, B, G, \mathcal{S}(q),  \mathcal{R}, K)$}
\label{alg:ssd}
\begin{algorithmic}
\STATE{Choose initial point $x_0 \in \R^n$}
\WHILE{$k \leq K$}
\STATE{From $\mathcal{S}(q)$, select matrix $S_{i_k}$ such that $i_k \in \mathcal{R}$. Then update
\begin{align*}
& x_{k+1}  = x_k - \alpha_{i_k} \ \nabla^{G} f_{i}(x_k); \quad \alpha_{i_k}   = \frac{\|\nabla^{G} f_{i}(x_k)\|^2_G}{\|\nabla^{G} f_{i}(x_k)\|^2_{Z_{i_k}}}
\end{align*}
$k \leftarrow k+1$;}
\ENDWHILE
\RETURN $x$
\end{algorithmic}
\end{algorithm}

\noindent Furthermore, if we take $G = B$, we get $Z_{i_k}  G^{-1} Z_{i_k} = Z_{i_k}  B^{-1} Z_{i_k} =  Z_{i_k}$. Then we get,
\begin{align*}
  x_{k+1}  = x_k - \nabla^{B} f_i(x_k), \quad  \alpha_{i_k} 
    = \frac{\|\nabla^{G} f_{i}(x_k)\|^2_G}{\|\nabla^{G} f_{i}(x_k)\|^2_{Z_{i_k}}}  =  \frac{\|x_k-x^*\|^2_{Z_{i_k}}}{\|x_k-x^*\|^2_{Z_{i_k}}} = 1 
\end{align*}
this can be interpreted as the so-called adaptive \textit{Sketch \& Project} method. This special framework was considered in \citep{richtrik2017stochastic, gower2019adaptive}. For comparison purpose with the SP method, we allow a parameter $\omega \in (0,2)$ in \eqref{jml9} to get the following:
\begin{align}
\label{jml13}
    x_{k+1} = x_k - \omega \alpha_{i_k} \nabla^{G} f_i(x);  \quad  \alpha_{i_k} =  \frac{\|\nabla^{G} f_{i}(x_k)\|^2_G}{\|\nabla^{G} f_{i}(x_k)\|^2_{Z_{i_k}}}
\end{align}

\subsection{Stochastic Steepest Descent with Momentum (SSDM)}
In this subsection, we introduce a momentum variant of the SSD algorithm. Introducing the so-called heavy ball update formula to the update of the SSD algorithm, we get the following:
\begin{align}
\label{jml16}
   x_{k+1} = x_k - \alpha_{i_k} \ \nabla^{G} f_i(x_k) + \gamma (x_k-x_{k-1}).
\end{align}

\begin{algorithm}
\caption{SSDM Algorithm: $x_{k+1} = \textbf{SSDM}(A,b,x_0, \gamma, B, G, \mathcal{S}(q),  \mathcal{R}, K)$}
\label{alg:SSDM}
\begin{algorithmic}
\STATE{Choose initial point $x_0 \in \R^n$}
\WHILE{$k \leq K$}
\STATE{From $\mathcal{S}(q)$, select matrix $S_{i_k}$ such that $i_k \in \mathcal{R}$. Then update
\begin{align*}
& x_{k+1}  = x_k - \alpha_{i_k}\ \nabla^{G} f_{i}(x_k) + \gamma (x_k-x_{k-1}); \quad \alpha_{i_k}   = \frac{\|\nabla^{G} f_{i}(x_k)\|^2_G}{\|\nabla^{G} f_{i}(x_k)\|^2_{Z_{i_k}}}
\end{align*}
$k \leftarrow k+1$;}
\ENDWHILE
\RETURN $x$
\end{algorithmic}
\end{algorithm}

\begin{remark}
Fix the following parameters: $G =I, \ B = A, \ S = A \succ 0$ along with an adaptive momentum parameter $\gamma_k = \alpha_k \frac{\beta_{k-1}}{\alpha_{k-1}}$ in the SSDM algorithm. Then, we get the following:
\begin{align}
     x_{k+1}= x_k + \alpha_k (b-Ax_k) + \alpha_k \frac{\beta_{k-1}}{\alpha_{k-1}} (x_k-x_{k-1})
\end{align}
Take, $u_k = b-Ax_k$ and $p_k = u_k + \frac{\beta_{k-1}}{\alpha_{k-1}} (x_k-x_{k-1})$. Then we can deduce the following:
\begin{align}
\label{conjugate}
  & x_{k+1} = x_k + \alpha_k p_k, \ \ u_{k+1} = u_k-\alpha_k Ap_k, \ \ \alpha_k =  \frac{p_k^T u_k}{\|p_k\|^2_{A}} =  \frac{\|u_k\|^2}{\|p_k\|^2_{A}}  \\
  & \beta_k = - \frac{p_k^T A u_{k+1}}{\|p_k\|^2_{A}} = \frac{\|u_{k+1}\|^2}{\|u_{k}\|^2}, \ \ p_{k+1} = u_{k+1} + \beta_k p_k
\end{align}
With the initial condition $u_0 = b-Ax_0, \ x_0 \in \R^n$, this is precisely the \textit{Conjugate Gradient Method} for solving linear system $Ax= b, \ A \succ 0$ \citep[see][]{bhaya2004}.
\end{remark}

\subsection{Greedy Sketching (GS)}
\label{ss}
We propose the following sketching rule: choose a sample of $\tau$ sketching matrices uniformly at random from the sketched matrix set $\mathcal{S}(q)$. Denote $\phi_k(\tau)$ as the generated index set by the above sampling process. At iteration $k$, set $i$ as
\begin{align}
\label{jml18}
i = \argmax_{i \in \phi_k(\tau)} f_i(x_k) = \frac{1}{2} \argmax_{i \in \phi_k(\tau)} \big \| Ax_k-b \big \|_{H_{i_k}}^2 = \frac{1}{2} \argmax_{i \in \phi_k(\tau)} r_k^T T_{i_k} r_k
\end{align}
To calculate the expectation with respect to this sketching rule, we will use the following setup. First, let us fix any random iterate $x_k$, then we sort the function values $f_{i}(x_k)$ from smallest to largest. Define, $f_{\underline{\mathbf{i_j}}}(x_k)$ as the $(\tau+j)^{th}$ entry on the sorted list, i.e.,
\begin{align}
\label{jml19}
  \underbrace{f_{\underline{\mathbf{i_0}}}(x_k)}_{\tau^{th}} \ \leq ... \leq \ \underbrace{f_{\underline{\mathbf{i_j}}}(x_k)}_{(\tau+j)^{th}} \ \leq ... \leq \ \underbrace{f_{\underline{\mathbf{i_{q-\tau}}}}(x_k)}_{q^{th}}.
\end{align}
Now, if we randomly choose any entry from list \eqref{jml19} then each entry has a equal probability of selection and the equal probability is $1/\binom{m}{\tau}$. Based on the discussion, we have the following:
{\allowdisplaybreaks
\begin{align}
\label{jml20}
 \E[f_i(x_k) \ | \ i \sim \mathcal{G(\tau)}] & = \frac{1}{ \binom{q}{\tau}} \sum\limits_{j = 0}^{q-\tau} \binom{\tau-1+j}{\tau-1} f_{\underline{\mathbf{i_j}}}(x_k) =  \frac{1}{2 \binom{q}{\tau}} \sum\limits_{j = 0}^{q-\tau} \binom{\tau-1+j}{\tau-1}  \|x_k-x^*\|_{Z_{\underline{\mathbf{i_j}}}}^2
 \end{align}}
This sketching approach allows us to connect two well-known sketching rules into one framework. For instance, with the choice $\tau =1$, we have the following:
{\allowdisplaybreaks
\begin{align}
\label{jml21}
 \E[f_i(x_k) \ | \ i \sim \mathcal{G}(\tau)] & = \frac{1}{ \binom{q}{1}} \sum\limits_{j = 0}^{q-1}  f_{\underline{\mathbf{i_j}}}(x_k) = \frac{1}{q} \sum\limits_{i = 1}^{q}  f_i(x_k).
 \end{align}}
This is the so-called uniform sketching rule. Similarly, by taking $\tau = m$, we get $\E[f_i(x_k) \ | \ i \sim \mathcal{G}(\tau)]  = \max_{i \in \{1,2,...,q\}} f_i(x_k)$. That is the so-called maximum distance sketching rule.

\subsection{Greedy Capped Sketching (GCS)}
\label{gs}
Now, using the GS rule, we will propose a greedy version of the capped sampling rule \citep{gower2019adaptive}. Pick two sampled sketching matrices of sizes $\tau_1$ and $\tau_2$ respectively, uniformly at random (sampling with replacement) from the set $\mathcal{S}(q)$. For any $0 \leq \theta \leq 1$, let's define the following:
\begin{align}
\label{jml25}
    \mathcal{W} = \left\{i  | \ f_i(x_k) \geq \theta  \E[f_i(x_k) \ | \ i \sim \mathcal{G}(\tau_1)] + (1-\theta) \E[f_i(x_k) \ | \ i \sim \mathcal{G}(\tau_2)] \right\}
\end{align}
then select $i \in \mathcal{W}$ with probability $p_i$. Set $\mathcal{W}$ is not empty as $\theta  \E[f_i(x_k) \ | \ i \sim \mathcal{G}(\tau_1)] + (1-\theta) \E[f_i(x_k) \ | \ i \sim \mathcal{G}(\tau_2)] \leq \max_{i \in \{1,2,..,q\}} f_i(x_k)$ holds. That implies $\max_{i \in \{1,2,..,q\}} f_i(x_k) \in \mathcal{W}$. The resulting expectation can be calculated \footnote{The expectation computation is not of practical choice for the implementation. We suggest to use a lower bound, i.e., $\E[f_i(x_k) \ | \ i \sim \mathcal{G}(\tau)]\} \geq \frac{1}{2 q} \sum \nolimits_{j=1}^{q} \|Ax_k-b\|^2_{H_{i_k}} $, see the proof of Theorem \ref{lem:mu} for details.} as follows:
\begin{align}
\label{jml26}
    \E[f_i(x_k) \ | \ i \sim \mathcal{C}(\theta, \tau_1,\tau_2)] & = \sum \limits_{j \in \mathcal{W}} p_j f_j(x_k)
\end{align}

\section{Convergence Theory}
\label{sec:conv}
In this section, we provide the convergence results for the proposed SSD and SSDM algorithms. Without loss of generality, we show the results for any generic sketching rule $\mathcal{R}$. For ease of presentation, we use the notation $\E_i[\cdot]$ to denote $\E_{i}[\cdot \ | \ i \sim \mathcal{R}]$ throughout this section. We start the section by discussing some useful Lemmas that will be used frequently in our analysis. The, we provide the convergence results of the following quantities: $\|\E[x_k-x^*]\|^2, \ \E[\|x_k-x^*\|_G^2], \ \E[f(x_k)]$, and $\E[f(\bar{x}_k)]$. We finish the section with the discussion about some special cases that can be obtained from the proposed methods.

\paragraph{Technical results} The following results are crucial for our convergence analysis. Similar types of results can be found in the literature that are used frequently to estimate the convergence rate of Steepest Descent methods \footnote{Note that, when implementing the algorithm, we only consider the case $\|x_k-x^*\|_{Z_{i_k}} > 0$ as $\|x_k-x^*\|_{Z_{i_k}} = 0$ implies $\nabla^G f(x_k) = 0 \ \Rightarrow \ x_{k+1} = x_k$.}.

\begin{lemma}
\label{lem:range}
If $x_0 \in \textbf{Range}(G^{-1}A^T)$, then we have the following:
\begin{align*}
  r_k \in \textbf{Range}(G^{-\frac{1}{2}}A^T) =  \textbf{Range} \left(\mathbf{T}\right) 
\end{align*}
\end{lemma}

\begin{lemma}
\label{lem:jm1}
Assume, $x^*$ is a solution of problem \eqref{jml2} and $x_k$ is a random iterate. Then, the following identities hold:
\begin{align}
   & \mu^+_1(i_k) \ r_k^T T_{i_k} r_k \leq r_k^T T^2_{i_k} r_k \leq  \mu_2(i_k) \ r_k^T T_{i_k} r_k,  \label{t:1}\\
    & \mu^+_1(i_k) \ r_k^T T^2_{i_k} r_k \leq r_k^T T^3_{i_k} r_k  \leq  \mu_2(i_k) \ r_k^T T^2_{i_k} r_k, \label{t:2} \\
    & \frac{1}{\mu_2} \leq \frac{1}{\mu_2(i_k)} \leq \alpha_{i_k} \leq   \frac{1}{\mu_1^+(i_k)} \leq \frac{1}{\mu^+_1}, \quad G = B \Rightarrow \alpha_{i_k} = 1 \label{t:4}
\end{align}
\end{lemma}

\begin{theorem}
\label{lem:mu}
For the greedy sketching rules proposed in subsections \ref{ss} and \ref{gs}, there exist constants $0 < \lambda_1^+ \leq \lambda_2 $ such that the following bound holds: 
\begin{align}
    & \lambda_1^+ \ \|r_k\|^2 \leq r_k^T \E[T] r_k = 2 f(x_k) \leq \lambda_2 \ \|r_k\|^2  \label{t:3}
\end{align}
where, the expectation is taken with respect to the corresponding sketching rules. Furthermore, we have the following estimation:
\begin{align}
    \textbf{GSR:} & \quad \lambda_1^+ =  \frac{ \lambda_{\min}^+\left(\mathbf{T}\right)}{\bar{q}_k(\tau)}, \ \ \lambda_2 =  \min \left\{\mu_2, \frac{ \tau \lambda_{\max}\left(\mathbf{T}\right)}{q}\right\} \label{g1}\\
    \textbf{GCS:} & \quad \lambda_1^+ =  \theta \frac{ \lambda_{\min}^+\left(\mathbf{T}\right)}{ \bar{q}_k(\tau_1)}  + (1-\theta) \frac{ \lambda_{\min}^+\left(\mathbf{T}\right)}{ \bar{q}_k(\tau_2)}, \ \ \lambda_2 =  \mu_2 \label{g2}
\end{align}
where, $\tau$ is the sketch sample size for GSR rule and $\tau_1$ and $\tau_2$ are the respective sketch sample sizes for the GCS rule. Moreover, $\bar{q}_k(\tau) = \max\{q-s_k, q- \tau+1\} \leq q$, with $s_k$ denotes the number of zero entries in the list $[f_1(x_k),f_2(x_k),...,f_q(x_k)]$.
\end{theorem}

\begin{remark}
Theorem \ref{lem:mu} suggests that the function $ f(x_k)$ (with the proposed sketching rules) has Lipschitz continuous gradient and strong convexity constant in the line segment $[x_k,x^*]$. Take, $x^*$ such that $Ax^* = b$ holds, then we have $f^* = f(x^*) = 0$ and $\nabla f(x^*) = 0 $. Therefore, the above-mentioned Lemmas can be restated as follows:
\begin{align}
    & \frac{\epsilon_1}{ 2} \|x_k-x^*\|^2_G  \ \leq \ f(x_k) - f^* - \langle \nabla f(x^*),x_k-x^* \rangle  \ \leq \  \frac{\epsilon_2}{2}\ \|x_k-x^*\|^2_G. \label{eq:strongm}
\end{align}
for some $\epsilon_1, \ \epsilon_2 > 0$. Furthermore, the following identity holds:
\begin{align}
\label{eq:rsi}
    \langle \E[\nabla^G f_i(x_k) \ | \ i \sim \mathcal{R}], x_k-x^* \rangle_G = 2 \E[f_{i}(x_k) \ | \ i \sim \mathcal{R}]\geq \epsilon_1 \|x_k-x^*\|^2_G.
\end{align}
for some generic sketching rule $\mathcal{R}$ (see Theorem \ref{lem:mu}).
\end{remark}

\begin{theorem}
\label{th:jm2}
Assume, $x^*$ is a solution of problem \eqref{jml2} and $x_k$ is a random iterate. Then, the following identities hold:
\begin{align}
    & 1 \leq \frac{(r_k^TT_{i_k}r_k)(r_k^T T_{i_k}^3 r_k)}{(r_k^T T_{i_k}^2 r_k)^2 }  \leq 1+ \frac{[\mu_2(i_k)-\mu_1(i_k)]^2}{4 [\mu_1^+(i_k)]^2} \leq 1+ \frac{\left(\sigma_{i_k}\right)^2}{4} \label{t:5}  \\
    & \textbf{If} \ \ T_{i_k} \succ 0 \ \ \forall i_k 	\Rightarrow \frac{(r_k^Tr_k)(r_k^T T_{i_k}^3 r_k)}{(r_k^T T_{i_k} r_k)(r_k^T T_{i_k}^2 r_k) }  \leq  \frac{[\mu_2(i_k)+\mu_1(i_k)]^2}{4 \mu_1(i_k) \mu_2(i_k)} =  \frac{\left(1+ \sigma_{i_k}\right)^2}{4\sigma_{i_k}} \label{t:6} \\
    & \textbf{If} \ \ T_{i_k} \succ 0 \ \ \forall i_k 	\Rightarrow \frac{(r_k^TT_{i_k}r_k)(r_k^T T_{i_k}^3 r_k)}{(r_k^T T_{i_k}^2 r_k)^2 }  \leq  \frac{[\mu_2(i_k)+\mu_1(i_k)]^2}{4 \mu_1(i_k) \mu_2(i_k)} =  \frac{\left(1+ \sigma_{i_k}\right)^2}{4\sigma_{i_k}} \label{t:7} 
\end{align}
\end{theorem}

\begin{remark}
The results of Theorem \eqref{th:jm2} are quite similar to the so-called Kantorovich inequality. However, in this specific case we don't assume the positive definiteness of the respective operator. For our case we only have positive semi-definiteness. 
\end{remark}

\begin{theorem}
\label{th:ssd1}
The sequence $x_k$ generated by SSD algorithm satisfies the following:
\begin{align*}
  \big \|\E[x_{k+1}-x^*] \big \|_{G}^2 \leq \lambda_{\max}^2\left(I - \omega \E\left[\alpha_{i_k}  T_{i_k} \ | \ i_k \sim \mathcal{R}\right]  \right) \big \|\E[x_{k}-x^*] \big \|_{G}^2 
\end{align*}
\end{theorem}

\begin{remark}
Since, for any random vector $x \in \R^n$, we have the following:
\begin{align}
  \|\E[x-x^*] \big \|_{G}^2 = \E \left[\|x- x^*\|_G^2\right] - \E \left[\big \|x- \E[x] \big \|_G^2\right]
\end{align}
We will provide the convergence of the term $\E \left[\|x- x^*\|_G^2\right]$ in Theorem \ref{th:jm4}.
\end{remark}

\begin{theorem}
\label{th:jm4}
If $ 0 < \omega <2$, then $x_k$ converges and the following results hold:
    \begin{align*}
    & T_{i_k} \succeq 0  \Rightarrow \ \E\left[\|x_{k+1}-x^* \|_{G}^2 \right]  \leq \left\{1- \frac{(2 \omega - \omega^2) \lambda^+_1}{\mu_2} \right\} \E\left[\|x_{k}-x^* \|_{G}^2 \right] \\
  & T_{i_k} \succ 0 \Rightarrow \ \E\left[\|x_{k+1}-x^* \|_{G}^2 \right]  \leq \left\{1- (2 \omega - \omega^2) \E \left[ \frac{ 4\sigma_{i} }{\left(1+ \sigma_{i}\right)^2} \right]\right\} \E\left[\|x_{k}-x^* \|_{G}^2 \right]
\end{align*}
Also the average iterate $\Tilde{x}_k = \sum \nolimits_{l=0}^{k-1} x_l$ for all $k \geq 1$ satisfies the following
\begin{align*}
    \E[\|\Tilde{x}_k-x^*\|^2_G]  \leq \frac{\mu_2 \|x_0-x^*\|^2_G }{ \omega k \lambda_1^+(2-\omega)}.
\end{align*}
\end{theorem}

\begin{remark}
\label{rem:ssd}
Consider the special case $G = B$. In that scenario, we have 
$\alpha_{i_k} = 1$ and $T^2_{i_k} = T_{i_k}$. Using these values in the first part of Theorem \ref{th:jm4}, we get the following:
\begin{align}
\E\left[\|x_{k+1}-x^* \|_{G}^2 \right]  \leq \left\{1- (2 \omega - \omega^2) \lambda^+_1 \right\} \E\left[\|x_{k}-x^* \|_{G}^2 \right] 
\end{align}
This is the result obtained by \citet{richtrik2017stochastic} for the special case of $G = B$. 
\end{remark}

\begin{theorem}
\label{th:ssd3}
If $ 0 < \omega <2$, then the following result holds for the function decay:
 \begin{align*}
& T_{i_k} \succeq 0  \Rightarrow \   \E \left[\frac{f_{i_k}(x_{k+1})}{f_{i_k}(x_{k})}\right]=  \E \left[\frac{\|x_{k+1}-x^*\|_{Z_{i_k}}^2}{\|x_{k}-x^*\|_{Z_{i_k}}^2}\right]  \leq 1-  \frac{4 (2 \omega - \omega^2) }{\min \{4 \E[\sigma_i], 4 + \E[\sigma_i^2]\}} \\
&  T_{i_k} \succ 0  \Rightarrow \  \E \left[\frac{f_{i_k}(x_{k+1})}{f_{i_k}(x_{k})}\right]=  \E \left[\frac{\|x_{k+1}-x^*\|_{Z_{i_k}}^2}{\|x_{k}-x^*\|_{Z_{i_k}}^2}\right] \leq 1- (2 \omega - \omega^2) \E \left[ \frac{ 4\sigma_{i} }{\left(1+ \sigma_{i}\right)^2} \right]
\end{align*}
Moreover, we have the following:
\begin{align*}
& T_{i_k} \succeq 0  \Rightarrow \ \E\left[f(x_{k+1}) \right]  \leq \frac{\lambda_2}{2} \left\{1- \frac{(2 \omega - \omega^2) \lambda^+_1}{\mu_2} \right\}^k \|x_0-x^* \|_{G}^2 \\
& T_{i_k} \succ 0 \Rightarrow \ \E\left[f(x_{k}) \right]  \leq \frac{\lambda_2}{2} \left\{1- (2 \omega - \omega^2) \E \left[ \frac{ 4\sigma_{i} }{\left(1+ \sigma_{i}\right)^2} \right]\right\}^k \|x_0-x^* \|_{G}^2
\end{align*}
Furthermore,  the average iterate $\Tilde{x}_k = \sum \nolimits_{l=0}^{k-1} x_l$ for all $k \geq 1$ satisfies the following
\begin{align*}
 \E[f(\Tilde{x}_k)] \leq  \frac{\mu_2 \|x_0-x^*\|^2_G }{2\omega k (2-\omega)}.
\end{align*}
\end{theorem}

\begin{remark}
\label{rem:specialresult}
Note that, the above Theorems are generalized results. As the constants varies form rules to rules, for different choices of sampling rules, we get the corresponding convergence results. In section \ref{sec:special}, we discuss some special algorithms, and their respective convergence results that can be obtained from the above Theorems.
\end{remark}

\begin{theorem}
\label{th:momentum1}
Choose $x_0 = x_1 \in \textbf{Range}(G^{-1}A^T)$. Let $x_k$ be the random iterate generated by the SSDM algorithm. Let, $x^*$ is the solution of problem \eqref{jml2} and $0 \leq \xi < \frac{\zeta  \mu^+_1}{\mu_2}$ such that the following quantities
\begin{align*}
    \phi_1 := 1 + 3 \gamma + 2 \gamma^2 -\frac{\left(\gamma  + 2- \omega-\zeta \right) \omega}{\mu_2} \lambda^+_1 , \quad \phi_2 := \gamma + 2 \gamma^2 + \frac{\left( \gamma  -\xi  \right) \omega}{\mu^+_1} \lambda_2
\end{align*}
satisfy the condition $\phi_1+ \phi_2 < 1$ for some $\gamma \geq \max \left\{\xi, \zeta-2+\omega\right\} $. Take, $0 \leq \delta = \max \{0, \frac{\xi \mu_2}{\zeta \mu^+_1}-\phi_1, \frac{-\phi_1+ \sqrt{\phi_1^2+4 \phi_2}}{2} \}$ and define the following Lyapunov function
\begin{align*}
     \mathcal{V}_{k} & \overset{\text{def}}{=}  \|r_k\|^2 + \delta \|r_{k-1}\|^2 + \frac{2 \zeta \omega}{\mu_2} \left[ f(x_{k-1})-f(x^*)\right] \\
     & = \|x_k-x^*\|^2_G + \delta \|x_{k-1}-x^*\|^2_G + \frac{2 \zeta \omega}{\mu_2} \left[ f(x_{k-1})-f(x^*)\right]
\end{align*}
Then, we have the following:
\begin{align*}
    \E\left[ \mathcal{V}_{k+1}\right] \leq \rho \E\left[ \mathcal{V}_{k}\right] \leq \cdots \leq \rho^k \E\left[ \mathcal{V}_{1}\right]   = \rho^k \left[(1+\delta)\|x_{0}-x^*\|^2_G+ \frac{2 \zeta \omega}{\mu_2} f(x_0)\right]
\end{align*}
where, $ 0 \leq \rho = \max \left\{\frac{\xi \mu_2}{\zeta \mu^+_1},\frac{\phi_1+\sqrt{\phi_1^2+ 4 \phi_2}}{2} \right\} < 1$,  and $\delta = \rho -\phi_1$. Furthermore, $\phi_1+ \phi_2  \leq \rho < 1$. Note that, $\E\left[ \mathcal{V}_{k}\right] \rightarrow 0$ implies $\E[\|x_k-x^*\|^2_G] \rightarrow 0$ and $\E\left[ f(x_{k-1})-f(x^*)\right] \rightarrow 0$.
\end{theorem}

\begin{remark}
\label{rem1}
Now, we will discuss Theorem \ref{th:momentum1} for the case of $G = B$ which was analyzed by \citet{loizou:2017}. From our earlier discussion, we deduce that $\alpha_{i_k} = \mu_1^+ = \mu_2 =1$ for this choice. Let's take, $\zeta, \xi \geq 0$ such that $\frac{\xi}{\zeta} < \frac{\lambda_1^+}{\lambda_2} \leq \frac{4 \xi}{\zeta}$ holds. Then, for any $0< \omega < \min \{2, \frac{1}{\zeta^2 (\lambda_1^+)^2}(4 \xi \lambda_2- \zeta \lambda_1^+)\}$, the condition $\phi_1+ \phi_2 < 1$ is satisfied if we choose momentum parameter $\gamma$ such that $\gamma \geq \max \left\{\xi, \zeta-2+\omega\right\} $ and
\begin{align*}
   &  \gamma <  \frac{-4+ \omega\lambda^+_{1}- \omega \lambda_2 +\sqrt{\left(4- \omega\lambda^+_{1}+ \omega \lambda_2 \right)^2 +16 \xi \omega \lambda_2+ 16 \omega \left( 2- \omega-\zeta\right) \lambda^+_1}}{8} = \gamma(\zeta, \xi)
\end{align*}
hold. Furthermore, it can be easily shown that the following relations
\begin{align*}
    & \phi_1(\zeta, \xi) + \phi_2(\zeta, \xi) \leq  \phi_1(0, 0) + \phi_2(0, 0) < 1, \quad \gamma(0,0) \leq \gamma (\zeta, \xi),  \quad \rho(\zeta, \xi) \leq \rho(0,0) < 1.
\end{align*}
hold. The rate $\rho(0,0)$ is obtained by \cite{loizou:2017} for this specific case. Furthermore, they also showed that for $0 < \gamma \leq \gamma(0,0)$ the algorithm converges. Considering the relations above we conclude that the obtained rate and $\gamma$ range in Theorem \ref{th:momentum1} are better than the exiting rate and $\gamma$ range.
\end{remark}

\begin{theorem}
\label{th:cesaro}
Let $\{x_k\}$ be the random sequence generated by the momentum algorithm. Let, $0 < \omega <2$ and $0 \leq \gamma < 1 $ such that the condition $ \frac{\omega}{\mu_2}+ \gamma \left(1+\frac{\mu_2}{\mu^+_1}\right)  < 2$ holds. Define, $\Tilde{x_k} = 1/k \sum \nolimits_{l =1}^{k}x_l$, then the following identity holds:
\begin{align*}
    \E \left[f(\bar{x}_k)\right] \leq \frac{ \mu^+_1 \mu_2 (1-\gamma)^2 \ \|x_0 -x^*\|_G^2 + 2 \gamma \omega \mu_2 f(x_0)}{2 \omega k \left(2\mu^+_1 \mu_2 - \gamma \mu^+_1 \mu_2 - \gamma \mu_2^2 -\omega \mu^+_1 \right)}.
\end{align*}
\end{theorem}

\section{Special Cases Discussion}
\label{sec:special}

In this section, we briefly mention how one can recover existing algorithms and their convergence results from the proposed algorithms and the convergence Theorems.

\paragraph{Steepest Descent Method.} Let, $A = B \succ 0, G =I, \ S = A, \ \omega = 1$. Then we have
\begin{align}
    Z = A S (S^TAB^{-1}AS)^{\dagger}S^T A = A,  \quad \nabla^G f_i(x_k) = Ax_k-b, \ \alpha_{i_k} = \frac{\|Ax_k-b\|^2}{\|Ax_k-b\|^2_A}
\end{align}

\begin{corollary}
\label{th:sd}
The following results hold for the SD algorithm:
\begin{align*}
& \|x_{k+1}-x^* \|^2\leq \left(\frac{ \sigma_{i}-1 }{ \sigma_{i}+1} \right)^2 \|x_{k}-x^* \|^2, \quad \text{and} \quad \|\Tilde{x}_k-x^*\|^2 \leq \frac{\sigma }{k} \|x_0-x^*\|^2. \\
& \|x_{k+1}-x^* \|^2_A\leq \left(\frac{ \sigma_{i}-1 }{ \sigma_{i}+1} \right)^2 \|x_{k}-x^* \|^2_A, \quad \text{and} \quad \|\Tilde{x}_k-x^*\|^2_A \leq \frac{\lambda_{\max}(A) }{k} \|x_0-x^*\|^2.
\end{align*}
where, $\Tilde{x}_k = \sum \nolimits_{l=0}^{k-1} x_l$.
\end{corollary}

\begin{corollary}
\label{th:sd1}
Let $\{x_k\}$ be the random sequence generated by the momentum algorithm. Let, $0 \leq \gamma < 1 $ such that the condition $ \gamma < \frac{2\lambda_{\max}(A) -1 }{(1+\sigma) \lambda_{\max}(A)}$ holds. Then the following identity holds:
\begin{align*}
    \|\Tilde{x}_k-x^*\|^2_A \leq \frac{ \sigma \lambda_{\min}(A) (1-\gamma)^2 \ \|x_0 -x^*\|^2 + \gamma \sigma \|x_k-x^*\|^2_A}{k \left(2\lambda_{\min}(A) \sigma - \gamma \lambda_{\min}(A) \sigma - \gamma \lambda_{\min}(A) \sigma^2 - 1\right)}.
\end{align*}
\end{corollary}

\begin{proof}
As, $G =I, \ B = A, \ S = A, \ \omega = 1$, we have $Z = A, \ 2f(x_k) = \|x_k-x^*\|^2_A$. Using these  values in the previous Theorems, we get the results of the above corollaries.
\end{proof}

\paragraph{Momentum Sampling Kaczmarz Motzkin (MSKM).} Take, $q = m, \ B =I, \ S_i= e_i$. Then, considering the greedy sketching rule of subsection \ref{ss} in SSDM method, we have
\begin{align}
\label{skm}
  x_{k+1} = x_k- \omega \frac{a_{i}^Tx_k -b_{i}}{\|a_{i}\|_{G^{-1}}^2} G^{-1}a_{i} + \gamma (x_k-x_{k-1}).   
\end{align}
where the index $i$ is chosen as $i =  \argmax_{i \in \phi_k(\tau)} |a_i^Tx_k-b_i|^2/\|a_i\|^2$ and $\phi_k(\tau)$ denotes the collection of $\tau$ rows chosen uniformly at random out of $m$ rows of the constraint matrix $A$. Considering the above in \eqref{g1}, \eqref{g2}, we can estimate the constant $\lambda_1^+$ and $\lambda_2$ as follows:
\begin{align*}
    \textbf{GSR:} & \quad \lambda_1^+ =  \frac{1}{m} \lambda_{\min}^+\left(G^{-\frac{1}{2}}A^TAG^{-\frac{1}{2}}\right) , \ \ \lambda_2 =  \min \left\{\mu_2, \frac{ \tau }{m} \lambda_{\max}\left(G^{-\frac{1}{2}}A^TAG^{-\frac{1}{2}}\right) \right\} \\
    \textbf{GCS:} & \quad \lambda_1^+ =  \frac{ 1}{m} \lambda_{\min}^+\left(G^{-\frac{1}{2}}A^TAG^{-\frac{1}{2}}\right), \ \ \lambda_2 =  \max_i \|a_i\|^2_{G^{-1}}
\end{align*}
where, $\mu_2 = \max_i \|a_i\|^2_{G^{-1}}$ (we assumed $\|a_i\|^2 =1$ for all $i$). Using the above parameters in 
Theorems \ref{th:ssd1}, \ref{th:jm4}, \ref{th:ssd3}, and \ref{th:momentum1}, we get the convergence results for the following special methods: $G = I, \ \gamma = 0, \ \tau = 1, m$: \citet{gower2019adaptive}, $G = I, \ \tau = 1$: \citet{loizou:2017} and $G = I, \ \gamma = 0$: \citet{haddock:2019}.

\paragraph{Momentum Sampling Stochastic Descent.} Take, $q = m$ in SSDM algorithm along with the greedy sketching rule. Then, considering $S_i= s_i$ as sketching vectors, we have
\begin{align}
\label{mscd}
& S_i = s_i, \ B \succ 0 : \ \   x_{k+1} = x_k- \omega \frac{s_i^T(Ax_k -b)}{\|A^Ts_i\|^2_{G^{-1} }} G^{-1}A^T s_{i} + \gamma (x_k-x_{k-1})
\end{align}
where the index $i$ is chosen as $i =  \argmax_{i \in \phi_k(\tau)} |s_i^T(Ax_k -b)|^2/\|A^Ts_i\|^2_{B^{-1}} $ and $\phi_k(\tau)$ denotes the collection of $\tau$ sketching vectors chosen uniformly at random out of $m$ sketching matrices. With the choice $B = A= G \succ 0, \ \tau = 1, m= n$ and $s_i = u_i$ (the eigenvectors of matrix $A$, i.e., $Au_i = \lambda_i u_i$ ), then the above method resolves into the following:
\begin{align}
 x_{k+1} = x_k- \omega \frac{\lambda_i u_i^Tx_k-u_i^Tb}{\lambda_i \|u_i\|^2_{G^{-1} }} G^{-1}u_{i} + \gamma (x_k-x_{k-1})
\end{align}
where, $i$ is chosen with probability $p_i = 1/m$. This is the so-called \textit{Stochastic Spectral Descent} method proposed by \citet{spectral}. With the above choice, we get
\begin{align}
    \textbf{GSR} \ \ \& \ \ \textbf{GCS:}  \quad \lambda_1^+ =  \frac{ 1}{m}, \ \ \lambda_2 =  \frac{ \tau}{m}, \ \ \mu_2 = 1
\end{align}
as $T_i = u_iu_i^T, \ \mathbf{T} = I$. Then considering Theorems \ref{th:ssd1}, \ref{th:jm4}, \ref{th:ssd3}, and \ref{th:momentum1}, we get new complexity results of the above method as well as the ones provided by \citet{spectral}.

\paragraph{Momentum Sampling Co-ordinate Descent (MSCD).} Take, $q = m$ in SSDM algorithm along with the greedy sketching rule. Then, we have the following:
\begin{align}
& S_i = e_i, \ B = A \succ 0 : \ \   x_{k+1} = x_k- \omega \frac{a_{i}^Tx_k -b_{i}}{\|Ae_i\|^2_{G^{-1} }} G^{-1}A e_{i} + \gamma (x_k-x_{k-1}) \label{mrcd1} \\
& S_i = Ae_i, \ B = A^TA \succ 0 : \ \  x_{k+1} = x_k- \omega \frac{A_i^T(Ax_k -b)}{\|A^TAe_i\|^2_{G^{-1} }} G^{-1}A^TA e_{i} + \gamma (x_k-x_{k-1})  \label{mrcd2}
\end{align}
where the index $i$ is chosen as $i =  \argmax_{i \in \phi_k(\tau)} |S_i^T(Ax_k -b)|^2/\|A^TS_i\|^2_{B^{-1}} $ and $\phi_k(\tau)$ denotes the collection of $\tau$ sketching vectors chosen uniformly at random out of $m$ vectors. Considering the choice \eqref{mrcd1} in \eqref{g1}, \eqref{g2}, we can estimate the constant $\lambda_1^+$ and $\lambda_2$ as follows:
\begin{align*}
    \textbf{GSR:} & \quad \lambda_1^+ =  \frac{1}{m} \lambda_{\min}^+\left(G^{-\frac{1}{2}}A^2G^{-\frac{1}{2}}\right) , \ \ \lambda_2 =  \min \left\{\mu_2, \frac{ \tau }{m} \lambda_{\max}\left(G^{-\frac{1}{2}}A^2G^{-\frac{1}{2}}\right) \right\} \\
    \textbf{GCS:} & \quad \lambda_1^+ =  \frac{ 1}{m} \lambda_{\min}^+\left(G^{-\frac{1}{2}}A^2G^{-\frac{1}{2}}\right), \ \ \lambda_2 =  \max_i \|a_i\|^2_{G^{-1}}
\end{align*}
where, $\mu_2 = \max_i \|a_i\|^2_{G^{-1}}$ (we assumed $A_{ii} =1$ for all $i$). Similarly, using the parameter choice of \eqref{mrcd2}, we get
\begin{align*}
    \textbf{GSR:} & \quad \lambda_1^+ =  \frac{1}{m} \lambda_{\min}^+\left(G^{-\frac{1}{2}}A^TADA^TAG^{-\frac{1}{2}}\right) ,  \ \lambda_2 =  \min \left\{\mu_2, \frac{ \tau }{m} \lambda_{\max}\left(G^{-\frac{1}{2}}A^TADA^TAG^{-\frac{1}{2}}\right) \right\} \\
    \textbf{GCS:} & \quad \lambda_1^+ =  \frac{ 1}{m} \lambda_{\min}^+\left(G^{-\frac{1}{2}}A^TADA^TAG^{-\frac{1}{2}}\right), \ \ \lambda_2 = \mu_2 =  \max_i \|A^TAe_i\|^2_{G^{-1}}/\|A_i\|^2
\end{align*}
where, $D$ is a $n \times m$ matrix such that $D_{ii} = 1/\|A_i\|^2, \ D_{ij} = 0 \ \forall i \neq j$. Using the above parameters in 
Theorems \ref{th:ssd1}, \ref{th:jm4}, \ref{th:ssd3}, and \ref{th:momentum1}, we get the convergence results for the following methods: $G = B = A, \ S_i = e_i \ \gamma = 0, \ \tau = 1, m$: \citet{lewis:2010,gower:2015, gower2019adaptive}, $G = B = A, \ S_i = e_i$: \citet{loizou:2017}.

\section{Numerical Experiments}
\label{sec:num}

In this section, we perform computational experiments to evaluate the performance of the proposed methods equipped with greedy sampling rules and momentum. We implement the proposed methods in \textit{MATLAB R2020a} at a workstation with 64GB RAM, Intel(R) Xeon(R) CPU E5-2670, two processors running at 2.30 GHz. For a fair understanding of the performance, we carry out the experiments on a wide range of datasets \footnote{The selected datasets has high condition number, the system $Ax = b$ is ill-conditioned.} such as 1) Gaussian system, and 2) LIBSVM data \citep{li:2016}, and 3) SuiteSparse Matrix \citep{suitesparse}. First, we fix $\omega = 1, \ G = B$ \footnote{$\omega =1$ has the best computational performance for both linear systems \citep{richtrik2017stochastic,loizou:2017} and linear feasibility problems \citep{haddock:2017, morshed2020generalization,morshed:momentum, morshed:sj}.}. Second, we select Greedy Kacamarz (GK) and Greedy Co-ordinate Descend (GCD) algorithms to test the performance of sampling rules and momentum. Finally, we set the following parameters: $\gamma = 0.1, 0.2, 0.3, 0.4, 0.5$ \footnote{The momentum parameter choice is arbitrary, we don't need the spectral information beforehand to select $\gamma$. Furthermore, all experiments were run for $10$ times and the averaged performance was reported.}, initial point, $x_0 = 1000*[1,1,...,1]^T $, halting residual error, $\|Ax-b\| \leq 10^{-10}$.

\paragraph{Test Datasets} For the GK method, random test instances are generated as follows: vector $x \in \R^{n}$ and matrix $A\in \R^{m \times n}$ are taken as i.i.d $\mathcal{N}(0,1)$, then $b$ is set as $b = Ax$ (we maintain the consistency of the system that way). We also consider the following ten datasets for LIBSVM collection: \texttt{sonar} ($200 \times 60$), \texttt{ionosphere} ($351 \times 34$), \texttt{australian} ($690 \times 14$), \texttt{breastcancer} ($683 \times 10$), \texttt{splice} ($1000 \times 60$), \texttt{svmguide3} ($1,243 \times 21$), \texttt{mushrooms} ($8124 \times 112$), \texttt{phishing} ($11055 \times 68$), \texttt{a7a} ($16,100 \times 123$), \texttt{a9a} ($32,561 \times 123$). The random instances of GCD method are generated as follows: we first generate Gaussian matrix $G \in \R^{m \times n}$ and vector $x$ and then set $A = G^TG, \ b = Ax$. For real dataset, we consider the following positive definite matrices from the SuiteSparse Matrix collection: \texttt{bcsstk01} ($48 \times 48$), \texttt{bcsstk02} ($66 \times 66$), \texttt{nos4} ($100 \times 100$). For both of these algorithms, we measure residual error ($\|Ax_k-b\|_2$) and relative error $\|x_k-x^{*}\|_B/\|x_0-x^*\|_B$ \footnote{We set $x^* = x_0 + A^{\dagger} (b-Ax_0)$, when $Ax = b$ has a unique solution $x^*= x_{int}$, where $x_{int}$ is the initial Gaussian vector used to generate $b$.} with respect to the number of iterations and CPU time measured by MATLAB tic-toc function. The rest of the section is divided into two subsections. In subsection 7.1, we compare the proposed sampling rules. In subsection 7.2, we discuss the effect of momentum on the greedy algorithms.

\subsection{Comparison among Sampling Rules without Momentum}
In this subsection, we perform comparison experiments for both methods with respect to the proposed sampling rules.
\begin{figure}[htbp]
 \vspace{- 10 pt}
  \centering
    \includegraphics[scale = 0.35]{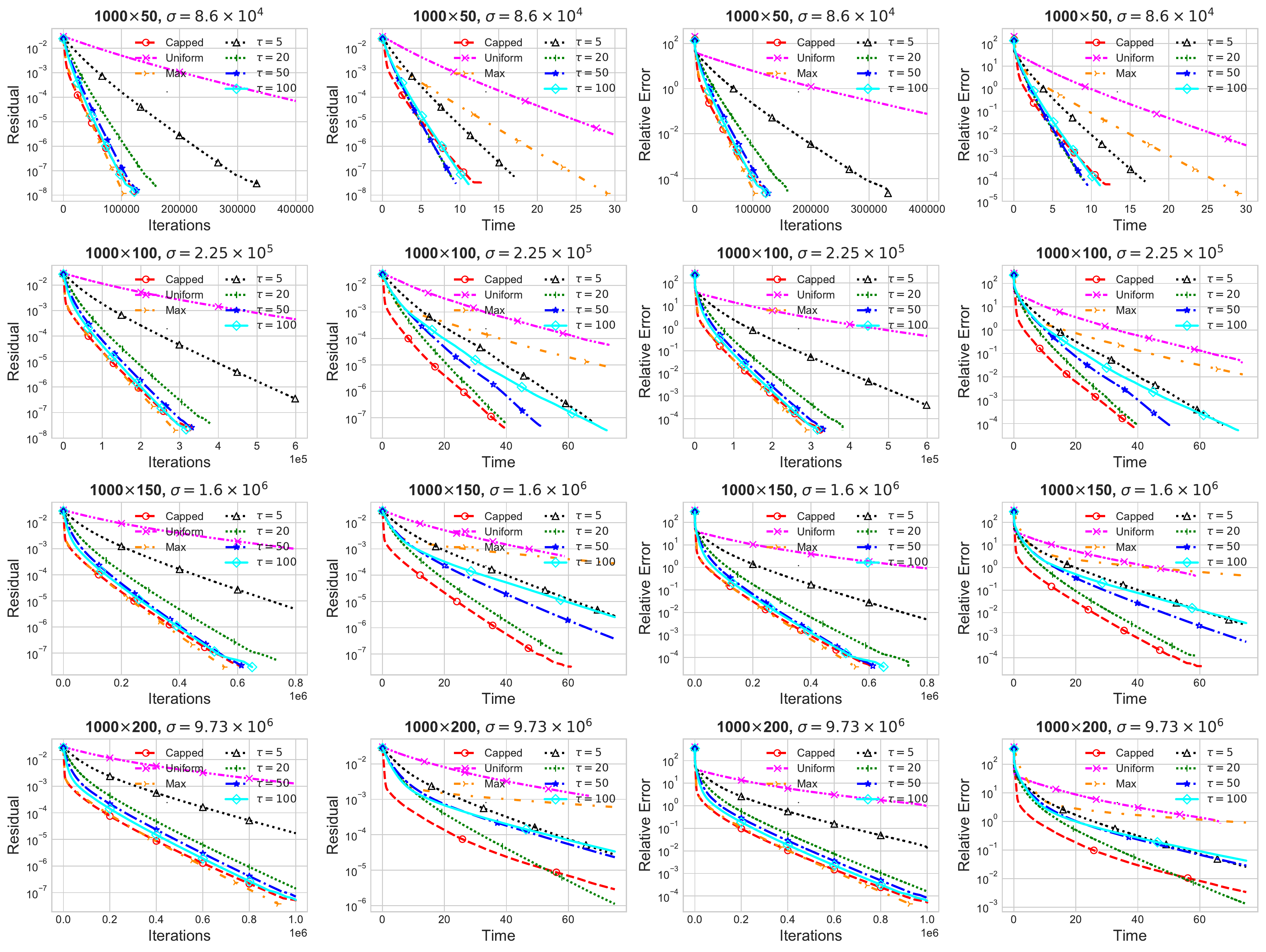}
    \caption{GK: sampling rules comparison on random data, left 2 panels: residual error vs No. of iterations and time, right 2 panels: relative error vs No. of iterations and time.}
    \label{fig:0}
    \vspace{- 10 pt}
\end{figure}

\begin{figure}[htbp]
 \vspace{- 10 pt}
  \centering
    \includegraphics[scale = 0.37]{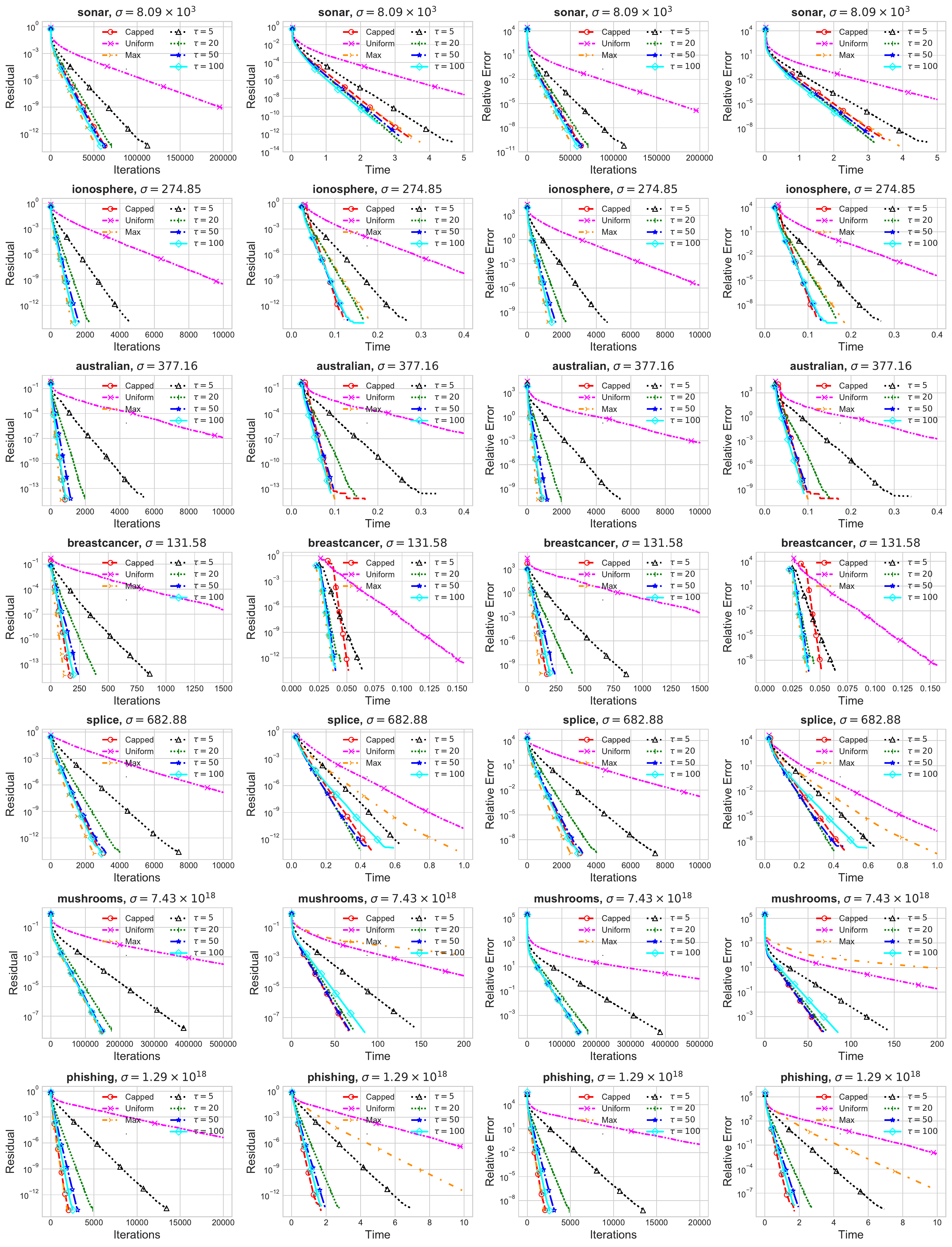}
    \caption{GK: sampling rules comparison on LIBSVM data, left 2 panels: residual error vs No. of iterations and time, right 2 panels: relative error vs No. of iterations and time.}
    \label{fig:1}
    \vspace{- 10 pt}
\end{figure}

\newpage

\begin{figure}[H]
 \vspace{- 10 pt}
  \centering
    \includegraphics[scale = 0.37]{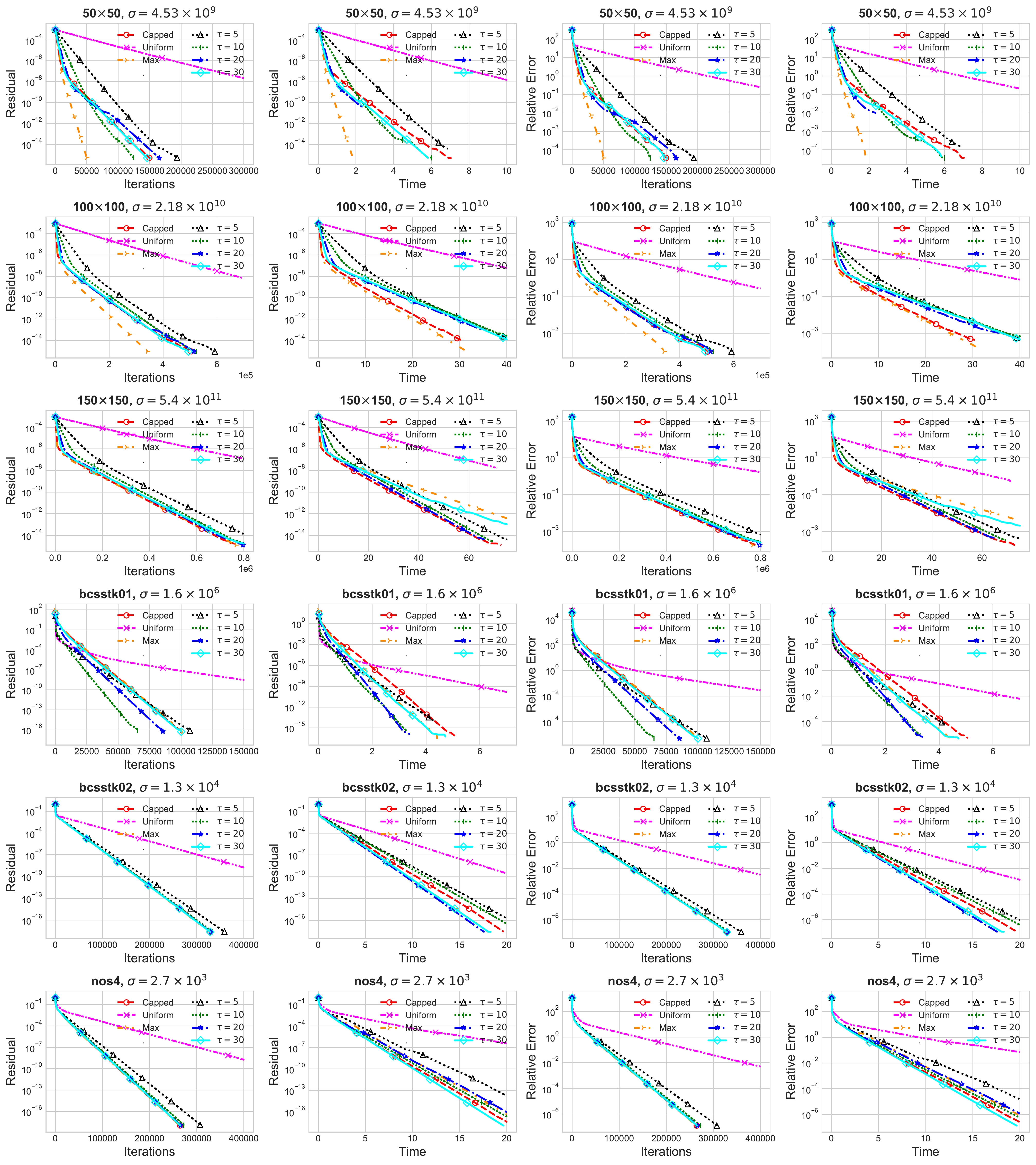}
    \caption{GCD: sampling rules comparison on random and real data, left 2 panels: residual error vs No. of iterations and time, right 2 panels: relative error vs No. of iterations and time.}
    \label{fig:2}
      \vspace{- 15 pt}
\end{figure}
We choose the following sampling rules (1) GK method: uniform ($\tau =1$), $\tau = 5$, $\tau = 20$, $\tau = 50$, $\tau = 100$, max. distance ($\tau = m$), capped ($\tau_1 = 1, \tau_2 = m, \theta = 0.5$), and (2) GCD method: uniform ($\tau =1$), $\tau = 5$, $\tau = 10$, $\tau = 20$, $\tau = 30$, max. distance ($\tau = m$), capped ($\tau_1 = 1, \tau_2 = m, \theta = 0.5$). We compare the sampling rules on 7 LIBSVM problems for the GK method, on 6 random/real problems for the GCD method. In Figures \ref{fig:0}, \ref{fig:1} and \ref{fig:2}, we plot the performance measures of the selected sampling rules for GK and GCD method, respectively. From the figures, it can be concluded that the proposed greedy sampling rules heavily outperform existing sampling rules such as uniform and max. distance rules for both GK and GCD. Overall, the uniform sampling rule has the worst performance compared to other sampling rules. Moreover, the greedy methods performs equally compared to each other (in the next subsection, we elaborate this point in more detail, also see the Appendix section for more comparison graphs).

\subsection{Comparison among Sampling Rules with Momentum}
First, we discuss the effect of momentum parameter $\gamma$ on the choice of greedy sampling parameter selection, i.e., $\tau$ and $\theta$. In Figures \ref{fig:3} and \ref{fig:4}, we plot the total CPU time taken by the momentum algorithms to reach a certain residual error threshold with varying $\tau$ and $\theta$. Here, we consider 10 LIBSVM problems for the GK method and 6 random/sparse problems for the GCD method. From the figures, it is evident that sample size $\tau$ has significant impact on the performance of momentum algorithms, whereas the effect of $\theta$ remains more or less constant for both methods. The choice $1 < \tau \ll m$ produces the best performing methods that confirms our claim about the importance of sampling. Furthermore, the proposed momentum variants heavily outperform the basic methods with no momentum. Since, sampling plays an important role, in the following we compare the momentum algorithms for fixed $\tau$.
\begin{figure}[htbp]
\centering
    \includegraphics[scale = 0.37]{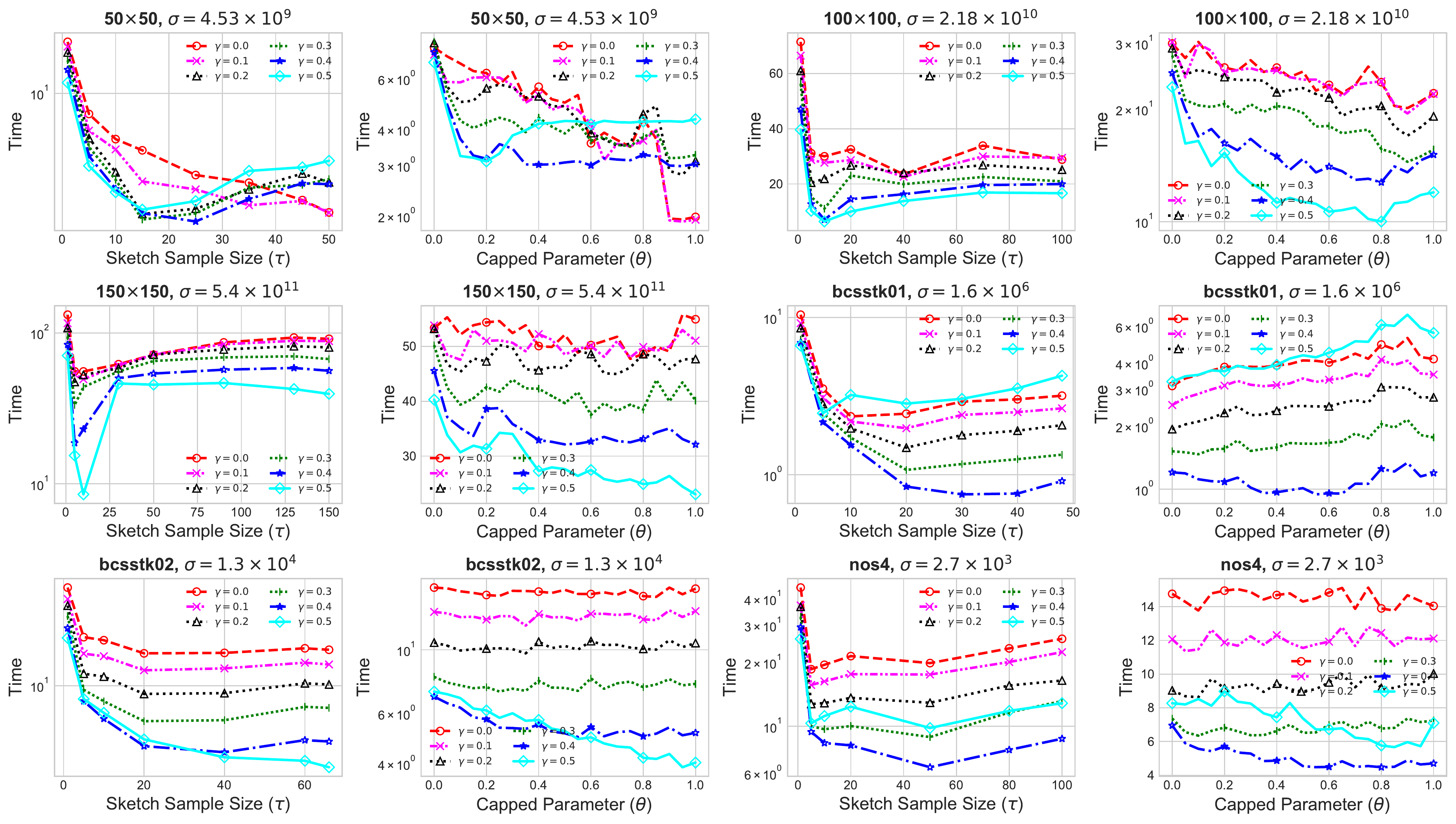}
    \caption{GCD momentum (effect of sketch sample size $\tau$ and capped parameter $\theta$).}
    \label{fig:3}
  \vspace{- 10 pt}
\end{figure}

\paragraph{GK and GCD method with momentum} Here we compare the momentum algorithms for some fixed sampling rules. In Figures \ref{fig:5}- \ref{fig:6}, we plot the comparison graphs for GK and GCD method, respectively with $\tau = 5$. From Figures \ref{fig:5} and \ref{fig:6}, we see that the proposed momentum variants outperform the basic GK and GCD methods for all of the considered test instances (please see Figures \ref{fig:55}-\ref{fig:17} for different $\tau$ in the Appendix section). In the following, we summarise the findings of our numerical experiments:
\newpage
\begin{figure}[H]
\centering
    \includegraphics[scale = 0.36]{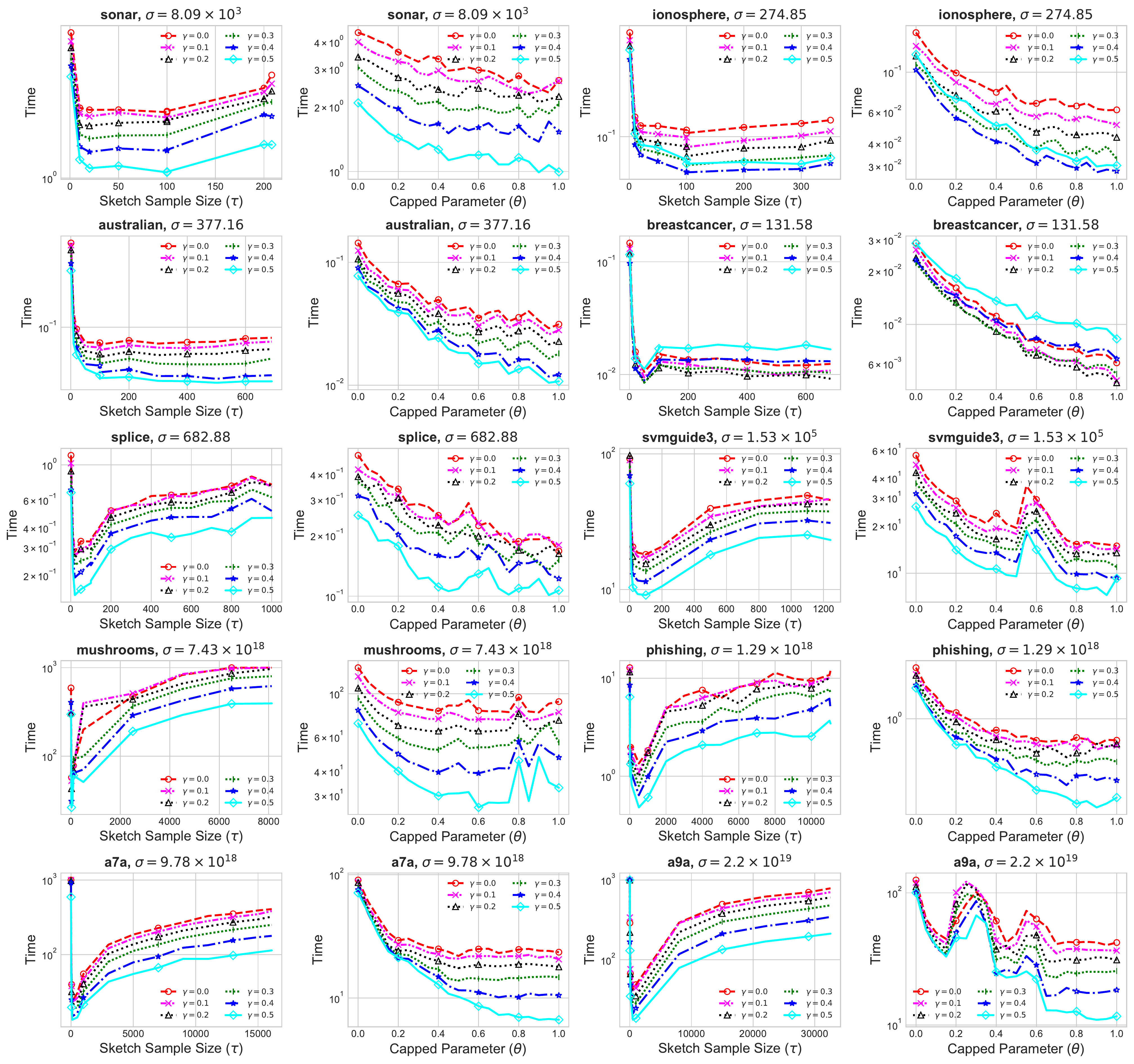}
    \caption{GK momentum (effect of sketch sample size $\tau$ and capped parameter $\theta$).}
    \label{fig:4}
  \vspace{- 10 pt}
\end{figure}
\begin{itemize}
    \item From our experiments, we find that the proposed greedy rules outperform the existing rules. For the case of greedy capped sampling, parameter $\theta$ has no significant impact on the algorithmic performance. However, for the greedy sketching rule, the choice of $\tau$ drives the performance and algorithms with sketch sample size $1 < \tau \ll m$ have the best performance. 
    
    \item Over the course of our experiments, $\gamma$ is chosen arbitrarily. It is evident that the choice $\gamma = 0.4$ leads to best performance in general. The choice $\gamma = 0.5$ works well most of the times but fails significantly from time to time. $\gamma = 0.4$ is the safe choice as it converges and has better performance.
    
    \item We test the proposed methods on a wide range of random/real-world datasets. We find that, the proposed methods performs better on ill-conditional linear systems (condition number of data matrix $A$ is large). We report the corresponding condition numbers in the figures.
\end{itemize}

\begin{figure}[htbp]
  \vspace{- 10 pt}
    \includegraphics[scale = 0.37]{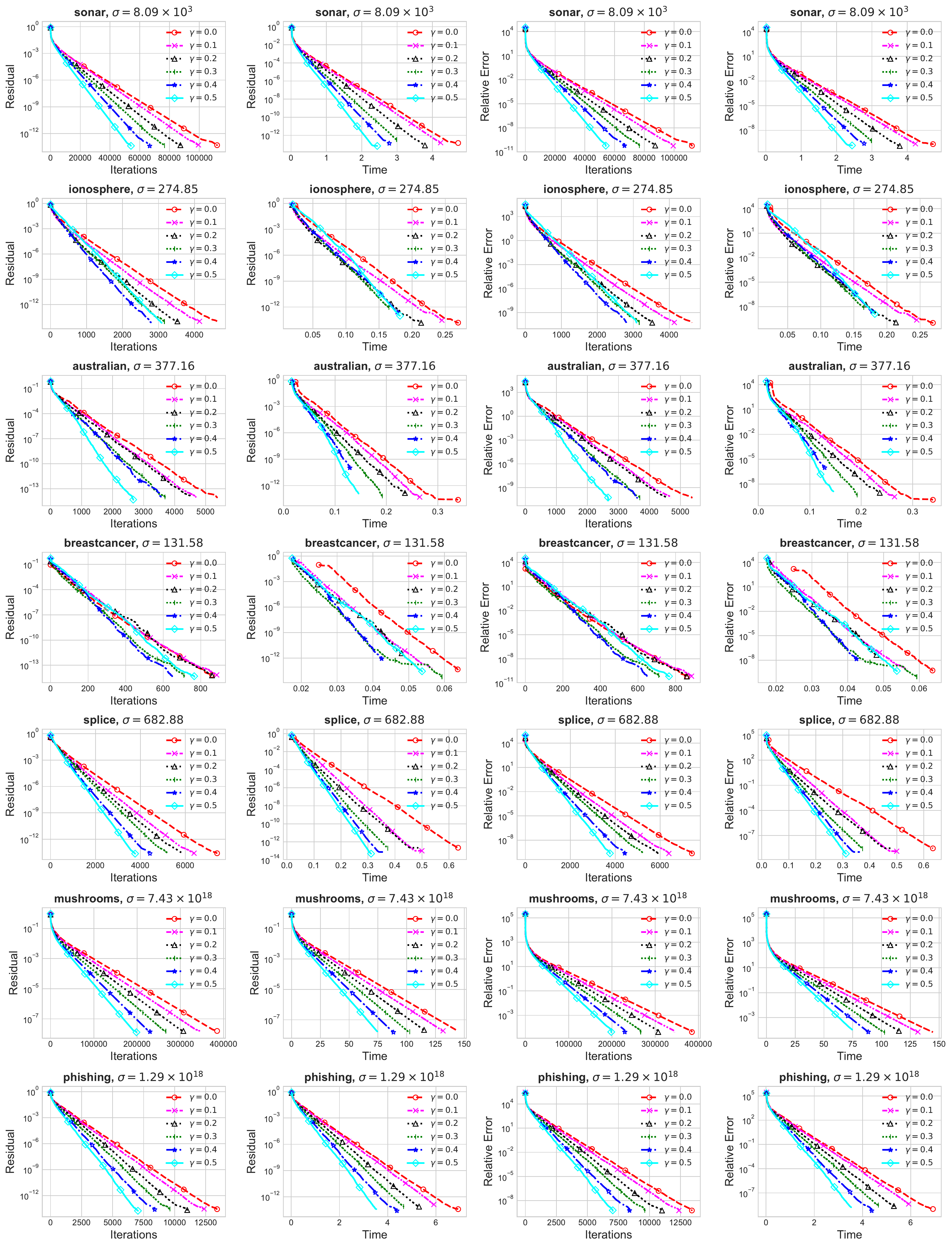}
    \caption{GK with momentum (sketch Sample size, $\tau = 5$): comparison among momentum variants on LIBSVM data, residual error and relative error vs CPU time and No. of iterations.}
    \label{fig:5}
      \vspace{- 10 pt}
\end{figure}
\newpage

\begin{figure}[htbp]
  \vspace{- 10 pt}
    \includegraphics[scale = 0.37]{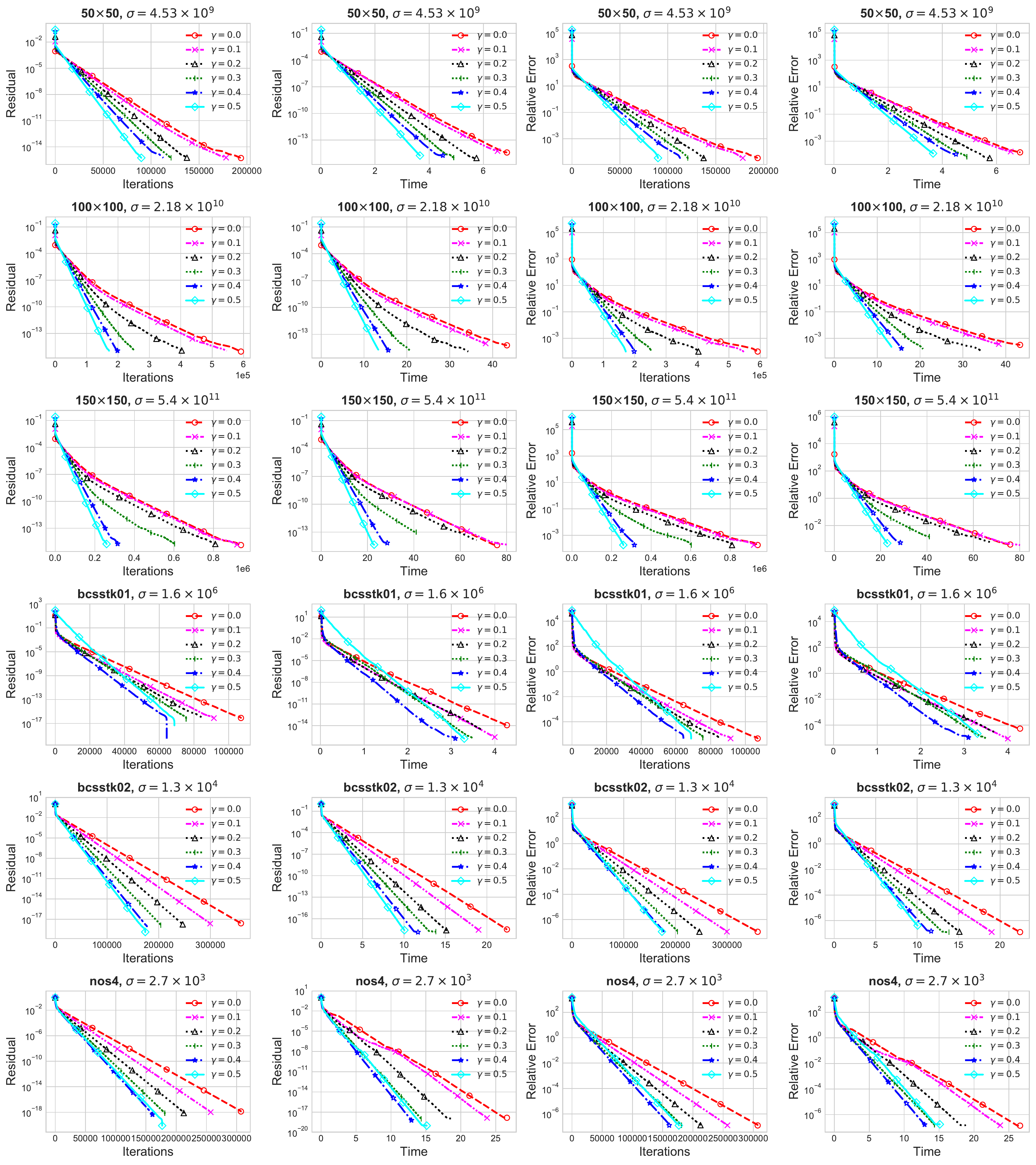}
    \caption{GCD with momentum (sketch Sample size, $\tau = 5$): comparison among momentum variants on Gaussian and Matrix market data, residual error and relative error vs CPU time and No. of iterations.}
    \label{fig:6}
      \vspace{- 10 pt}
\end{figure}

\section{Conclusions}
\label{sec:conclusions}

In this work, we propose a stochastic steepest descent framework for solving linear system of equations. Furthermore, we propose to incorporate greedy sketching strategies along with momentum technique to the basic method. In doing so, we synthesize well-known iterative methods such as steepest descent and randomized iterative methods such as randomized Kaczmarz, randomized co-ordinate descent into one framework. From our convergence analysis one can recover multiple convergence convergence results by varying sketching rules and sketching matrices. We validate the proposed algorithmic variants on a wide variety of datasets such as random Gaussian, LIBSVM, and matrix-market sparse matrices. From our numerical experiments, we conclude that the proposed sketching rule based methods heavily outperform the existing sketching rules. Moreover, the proposed momentum scheme accelerate algorithmic performance of the basic method. We conclude the paper with the following research query: we propose a connection between randomized iterative methods with iterative methods for solving linear system, one could ask if there exists a randomized scheme based on the SSDM method that will synthesize the well-known \textit{Conjugate Gradient} into an equivalent randomized iterative method framework. Another possible extension would be to develop extensions of the proposed greedy sketching rules based on adaptive sketch sample size, i.e., $\tau_k$.

\section*{Appendix}




\appendix
\begin{lemma}
\label{lem:skmseq}
(Lemma 2.1 in \citep{haddock:2017})  Let $\{x_k\},\ \{y_k\}$ be real non-negative sequences such that $x_{k+1} > x_k > 0$ and $y_{k+1} \geq y_k \geq 0$, then
\begin{align*}
  \sum\limits_{k=1}^{n} x_k y_k \ \geq \   \sum\limits_{k=1}^{n} \overline{x} y_k, \quad \text{where} \ \ \overline{x} = \frac{1}{n}\sum\limits_{k=1}^{n} x_k.
\end{align*}
\end{lemma}

\begin{lemma}
\label{lem:strang}
(\citep{strang1960}) For any matrix $M \succ 0$, we have the following bound:
\begin{align*}
    \|u\|^2_M \|v\|^2_{M^{-1}} \leq \frac{\left[\lambda_{\max}(M)+ \lambda_{\min}(M)\right]^2}{4 \lambda_{\min}(M) \lambda_{\max}(M)} \|u\|^2 \|v\|^2
\end{align*}
for all $u,v \in \R^n$ and this is the best bound possible.
\end{lemma}

\begin{lemma}
\label{lem:gower}
(Lemma 14 in \citep{gower2019adaptive}) If the relation $\textbf{Null}(P) \subset \textbf{Null}(M^T)$ holds for some matrix $M$ and positive semi-definite matrix $P \succeq 0$. Then, we have
\begin{align*}
\textbf{Null}(M) = \textbf{Null}(M^TPM), \quad  \textbf{Range}(M^T) = \textbf{Range}(M^TPM)
\end{align*}
\end{lemma}

\section{Proofs}
\label{sec:prf}

\begin{proof} (Lemma \ref{lem:range})
From the optimization problem \eqref{jml2}, we find that
\begin{align*}
    x^* = - G^{-1}A^T \lambda \in \textbf{Range}(G^{-1}A^T)
\end{align*}
for some $\lambda \in \R^m$. From the definition, $x_0 \in \textbf{Range}(G^{-1}A^T)$. Thus $r_0 \in \textbf{Range}(G^{-\frac{1}{2}}A^T)$. Now, assume that this identity holds for the first $k$ iterates. i.e.,
\begin{align*}
    r_{k-1}  \in \textbf{Range}(G^{-\frac{1}{2}}A^T), \quad r_k  \in \textbf{Range}(G^{-\frac{1}{2}}A^T)
\end{align*}
Now, from the momentum update formula, we have the following:
\begin{align*}
    r_{k+1} = r_k - \omega  \underbrace{G^{-\frac{1}{2}} A^T H(Ax-b)}_{ \in \ \textbf{Range}(G^{-\frac{1}{2}}A^T)} + \gamma \underbrace{(r_k-r_{k-1})}_{ \in \ \textbf{Range}(G^{-\frac{1}{2}}A^T)}
\end{align*}
Which implies $r_{k+1} \in \textbf{Range}(G^{-\frac{1}{2}}A^T)$. This proves the first part. Note, that
\begin{align*}
    \textbf{Null}(\mathbf{T}) = \textbf{Null} \left(G^{-\frac{1}{2}}A^T \sum \limits_{i=1}^{q} H_i A G^{-\frac{1}{2}}\right) =  \textbf{Null} \left(A^T \sum \limits_{i=1}^{q} H_i A G^{-\frac{1}{2}}\right) =  \textbf{Null} \left( A G^{-\frac{1}{2}}\right) 
\end{align*}
in the last equality, we used Lemma \ref{lem:gower} with $P = \sum \nolimits_{i=1}^{q} H_i, \ M = A$. Taking orthogonal complement, we have the required relation.
\end{proof}

\begin{proof} (Lemma \ref{lem:jm1})

We will first prove the first three identities. The upper bounds of \eqref{t:1},  \eqref{t:2}, and \eqref{t:3} follows from the fact that $T_{i_k}$ and $\E[T]$ are positive semi-definite matrices. Note that, as $T_{i_k}^{\frac{1}{2}} r_k \in \textbf{Range}(T_{i_k}^{\frac{1}{2}})$, by the Courant-Fisher Theorem, we have the lower bound of \eqref{t:1}. Similarly the lower bound of \eqref{t:2} follows from the fact that $T_{i_k}r_k \in \textbf{Range}( G^{-\frac{1}{2}} Z_{i_k}^{\frac{1}{2}})$. These arguments prove the first two identities of the Lemma. Now we have,
\begin{align}
\label{jm:10}
   \frac{1}{\mu_2} \leq  \frac{1}{\mu_2(i_k)}  \leq  \alpha_{i_k} = \frac{r_k^TT^2_{i_k} r_k}{r_k^TT^3_{i_k} r_k}  \leq   \frac{1}{\mu^+_1(i_k)} \leq \frac{1}{\mu^+_1}
\end{align}
Furthermore, with the choice $G = B$, we get $T^2_{i_k} = T_{i_k}$. That implies $\alpha_{i_k} = 1$. This proves the Lemma.
\end{proof}

\begin{proof} (Theorem \ref{lem:mu}) Using the definition of expectation from \eqref{jml20}, we have
{\allowdisplaybreaks
\begin{align}
\label{jml22}
 \E_i[f_i(x_k)] & =  \frac{1}{2 \binom{q}{\tau}} \sum\limits_{j = 0}^{q-\tau} \binom{\tau-1+j}{\tau-1}  \|x_k-x^*\|_{Z_{\underline{\mathbf{i_j}}}}^2  \leq \frac{\binom{q-1}{\tau-1}}{ 2\binom{q}{\tau}} \sum\limits_{j = 0}^{q-\tau}   r_k^T T_{\underline{\mathbf{i_j}}} r_k \nonumber \\
 & \leq \frac{\tau}{2 q} \sum\limits_{j = 1}^{q}  r_k^T T_j r_k  = \frac{\tau}{2 q} r_k^T \mathbf{T}r_k \leq \frac{ \tau \lambda_{\max}\left(\mathbf{T}\right)}{2q} \|r_k\|^2
 \end{align}}
Furthermore, the following holds
{\allowdisplaybreaks
\begin{align}
\label{jml23}
 \E_i[f_i(x_k) ] & =  \frac{1}{ \binom{q}{\tau}} \sum\limits_{j = 0}^{q-\tau} \binom{\tau-1+j}{\tau-1}  f_{\underline{\mathbf{i_j}}}(x_k)  \leq \max_{i \in \{1,2,...,q\}} f_i(x_k) \frac{1}{ \binom{q}{\tau}} \sum\limits_{j = 0}^{q-\tau} \binom{\tau-1+j}{\tau-1} \nonumber \\
 & = \frac{1}{2} \max_{i \in \{1,2,...,q\}} r_k T_i r_k \leq \frac{\max_{i \in \{1,2,...,q\}} \mu_2(i)}{2} \|r_k\|^2 \leq \frac{ \mu_2}{2} \|r_k\|^2 
 \end{align}}
here, we used column-sum property of Pascal’s triangle, i.e., $\sum\nolimits_{j = 0}^{q-\tau} \binom{\tau-1+j}{\tau-1} =  \binom{q}{\tau}$. Combining \eqref{jml22} and \eqref{jml23}, we get the upper bound of the proposed Lemma. Similarly, we have
{\allowdisplaybreaks
\begin{align}
\label{jml24}
 \E_i&[f_i(x_k)  ]  =  \frac{1}{2 \binom{q}{\tau}} \sum\limits_{j = 0}^{q-\tau} \binom{\tau-1+j}{\tau-1}  \|x_k-x^*\|_{Z_{\underline{\mathbf{i_j}}}}^2 \overset{\text{Lemma} \ \ref{lem:skmseq}}{\geq} \frac{1}{2 \binom{q}{\tau}} \sum\limits_{j = 0}^{q-\tau}  \frac{\sum\limits_{l = 0}^{q-\tau} \binom{\tau-1+l}{\tau-1}}{q-\tau +1} \|x_k-x^*\|_{Z_{\underline{\mathbf{i_j}}}}^2 \nonumber \\
 & = \frac{1}{2(q-\tau+1)} \sum\limits_{j = 0}^{q-\tau} \|x_k-x^*\|_{Z_{\underline{\mathbf{i_j}}}}^2  \geq \frac{1}{2 (q-\tau+1)} \min \left\{1, \frac{q-\tau+1}{q-s_k}\right\} \sum\limits_{j = 1}^{q} \|x_k-x^*\|_{Z_j} \nonumber \\
 &  = \frac{1}{2 \bar{q}_k} (x_k-x^*)^T G^{\frac{1}{2}} \mathbf{T} G^{\frac{1}{2}}(x_k-x^*)   \geq   \frac{ \lambda_{\min}^+\left(\mathbf{T}\right)}{2 \bar{q}_k}  \|r_k\|^2
 \end{align}}
 the last inequality follows from Lemma \ref{lem:range} along with the Courant-Fisher theorem. Similarly, for the capped rule, we have
 \begin{align}
 \label{cap1}
    \E[f_i(x) \ | \ i \sim \mathcal{C}(\theta, \tau_1,\tau_2)] & = \sum \limits_{j \in \mathcal{W}} p_j f_j(x) \geq \theta  \E[f_j(x) \ | \ j \sim \mathcal{G}(\tau_1)] + (1-\theta) \E[f_j(x) \ | \ j \sim \mathcal{G}(\tau_2)] \nonumber \\
    & \geq \theta \frac{ \lambda_{\min}^+\left(\mathbf{T}\right)}{ 2\bar{q}_k(\tau_1)} \|r_k\|^2  + (1-\theta) \frac{ \lambda_{\min}^+\left(\mathbf{T}\right)}{ 2\bar{q}_k(\tau_2)} \|r_k\|^2.
\end{align}
Similarly, we have
 \begin{align}
 \label{cap2}
    \E[f_i(x) \ | \ i \sim \mathcal{C}(\theta, \tau_1,\tau_2)]  = \sum \limits_{j \in \mathcal{W}} p_j f_j(x) \leq \max_{i \in \{1,2,...,q\}} f_i(x) \leq \frac{\mu_2}{2} \|r_k\|^2.
\end{align}
Combining \eqref{cap1} and \eqref{cap2}, we get the required result.
\end{proof}

\begin{proof} (Theorem \ref{th:jm2})
To prove the bound of \eqref{t:5}, we use the Cauchy-Swartrz inequality, i.e,
\begin{align}
\label{jm11}
  (r_k^TT^2_{i_k} r_k)^2 \leq \|T_{i_k}^{\frac{1}{2}}r_k\|^2  \|T_{i_k}^{\frac{3}{2}}r_k\|^2   = (r_k^TT_{i_k} r_k)(r_k^TT^3_{i_k} r_k)
\end{align}
Since, $T_{i_k} \succeq 0$, the relation $\mu_1(i_k) \|r_k\|^2 \leq r_k^TT_{i_k} r_k \leq \mu_2(i_k) \|r_k\|^2$ holds. Now, let's denote 
\begin{align*}
  \mathcal{H}(r_k) = (\mu_1(i_k)+\mu_2(i_k))r_k^TT_{i_k} r_k-r_k^TT^2_{i_k} r_k-\mu_1(i_k) \mu_2(i_k) \|r_k\|^2  
\end{align*}
Note that, considering the given conditions we can check that $\mathcal{H}(r_k) \geq 0$ holds for all $k \geq 1$. Then, we have the following:
\begin{align}
\label{jm12}
    (r_k^TT^2_{i_k} r_k) \|r_k\|^2 - (r_k^TT_{i_k} r_k)^2 & = \left(r_k^TT_{i_k} r_k-\mu_1(i_k) \|r_k\|^2\right) \left(\mu_2(i_k) \|r_k\|^2- r_k^TT_{i_k} r_k\right) - \|r_k\|^2 \mathcal{H}(r_k) \nonumber \\
    &  \leq \|r_k\|^4 \left(\frac{r_k^TT_{i_k} r_k}{\|r_k\|^2}-\mu_1(i_k)\right) \left(\mu_2(i_k)-\frac{r_k^TT_{i_k} r_k}{\|r_k\|^2}\right) \nonumber \\
    & \leq \|r_k\|^4 \ \frac{[\mu_2(i_k)-\mu_1(i_k)]^2}{4}
\end{align}
where, the last inequality follows from the fact that $\mu_1(i_k) \leq r_k^TT_{i_k} r_k/\|r_k\|^2 \leq \mu_2(i_k)$ holds. Now, replacing $r_k$ with $T_{i_k}^{\frac{1}{2}}r_k$ in equation \eqref{jm12} and simplifying further, we get
\begin{align}
\label{jm13}
    \frac{(r_k^TT_{i_k} r_k)(r_k^TT^3_{i_k} r_k)}{(r_k^TT^2_{i_k} r_k)^2 } - 1 \leq \frac{[\mu_2(i_k)-\mu_1(i_k)]^2}{4} \frac{(r_k^TT_{i_k} r_k)^2}{(r_k^TT^2_{i_k} r_k)^2} \leq \frac{[\mu_2(i_k)-\mu_1(i_k)]^2}{4 [\mu_1^+(i_k)]^2} \leq \frac{\left(\sigma_{i_k}\right)^2}{4}
\end{align}
where, the last inequality follows from the fact that $\mu_1(i_k) \geq 0$. Combining \eqref{jm11} and \eqref{jm13}, we get the result of \eqref{t:5}
Now, assume $T_{i_k} \succ 0$ for each $i_k$. This implies $0 < \mu_1(i_k) = \mu_1^+(i_k)$ and $T^{-1}(i_k)$ exists for each $i_k \geq 1$. Now, denote $M = T_{i_k}, \ u = T_{i_k} r_k, \ v = T^{\frac{1}{2}}_{i_k} r_k$. Substituting the above parameter values in Lemma \ref{lem:strang}, we get the following:
\begin{align*}
 \frac{ \|u\|^2_M \|v\|^2_{M^{-1}} }{\|u\|^2 \|v\|^2} =  \frac{ \|T_{i_k} r_k\|^2_{T_{i_k}} \|T^{\frac{1}{2}}_{i_k} r_k\|^2_{T_{i_k}^{-1}} }{\|T_{i_k} r_k\|^2 \|T^{\frac{1}{2}}_{i_k} r_k\|^2} = \frac{(r_k^Tr_k)(r_k^T T_{i_k}^3 r_k)}{(r_k^T T_{i_k} r_k)(r_k^T T_{i_k}^2 r_k) }  \leq  \frac{[\mu_2(i_k)+\mu_1(i_k)]^2}{4 \mu_1(i_k) \mu_2(i_k)} =  \frac{\left(1+ \sigma_{i_k}\right)^2}{4\sigma_{i_k}}
\end{align*}
this proves the bound of \eqref{t:6}. In a similar fashion, take $M = T_{i_k}, \ u = v= T_{i_k} r_k$ in Lemma \ref{lem:strang}. That gives us,
\begin{align*}
  \frac{ \|u\|^2_M \|v\|^2_{M^{-1}} }{\|u\|^2 \|v\|^2} =  \frac{ \|T_{i_k} r_k\|^2_{T_{i_k}} \|T_{i_k} r_k\|^2_{T_{i_k}^{-1}} }{\|T_{i_k} r_k\|^4} =   \frac{(r_k^TT_{i_k}r_k)(r_k^T T_{i_k}^3 r_k)}{(r_k^T T_{i_k}^2 r_k)^2 }  \leq  \frac{[\mu_2(i_k)+\mu_1(i_k)]^2}{4 \mu_1(i_k) \mu_2(i_k)} =  \frac{\left(1+ \sigma_{i_k}\right)^2}{4\sigma_{i_k}} 
\end{align*}
This is precisely what we claimed in \eqref{t:7}. 
\end{proof}

\begin{proof} (Theorem \ref{th:ssd1})
From the update formula we have,
  \begin{align*}
    x_{k+1} - x^*  = x_k - x^* + \omega \alpha_{i_k} G^{-1} Z_{i_k} (x^*-x_k) = \left(I - \omega \alpha_{i_k} G^{-1} Z_{i_k} \right) (x_k-x^*)
  \end{align*}
Taking expectation with respect to the index ${i_k}$, we get
conditioned on $x_k$, we get,
\begin{align*}
   \E_{i_k}\left[ x_{k+1} - x^*\right] = \left(I - \omega G^{-1} \E_{i_k}\left[\alpha_{i_k}  Z_{i_k}\right] \right) (x_k-x^*)   
\end{align*}
Taking expectation again and using tower property we get
\begin{align*}
   \E\left[ x_{k+1} - x^*\right]   = \E\left[ \E_{i_k}\left[ x_{k+1} - x^* \right] \right]  &  = \E \left[\left(I - \omega G^{-1} \E_{i_k}\left[\alpha_{i_k}  Z_{i_k}\right] \right) (x_k-x^*) \right] \\
   & = \left(I - \omega G^{-1} \E_{i_k}\left[\alpha_{i_k}  Z_{i_k}\right] \right) \E \left[x_k-x^*\right]   
\end{align*}
Taking $G-$ norm in both sides, we get the following
\begin{align*}
   \big \|\E\left[ x_{k+1} - x^*\right] \big \|_G^2 & = \big \| \left(I - \omega G^{-1} \E_{i_k}\left[\alpha_{i_k}  Z_{i_k}\right] \right) \E \left[x_k-x^*\right]   \big \|_G^2   \\
   & \leq \big \| I - \omega G^{-1} \E_{i_k}\left[\alpha_{i_k}  Z_{i_k} \right]  \big \|_G^2 \big \|\E\left[ x_{k} - x^*\right] \big \|_G^2
\end{align*}
Now, from the definition of $G-$ norm we get
\begin{align*}
  \big \| I - \omega G^{-1} & \E_{i_k}[\alpha_{i_k}  Z_{i_k} ]  \big \|_G^2  = \max_{\|G^{\frac{1}{2}}u\|_2 = 1}  \big \| \left( I - \omega G^{-1} \E_{i_k}\left[\alpha_{i_k}  Z_{i_k} \right]   \right) u \big \|_2^2 \\
  & = \max_{\|v\|_2 = 1}  \big \| \left(I -\omega G^{-\frac{1}{2}} \E_{i_k}\left[\alpha_{i_k}  Z_{i_k} \right] G^{-\frac{1}{2}} \right) v \big \|_2^2  = \lambda_{\max}^2\left(I -\omega \E_{i_k}\left[\alpha_{i_k}  T_{i_k}\right]  \right)
\end{align*}
This proves the Theorem.
\end{proof}

\begin{proof} (Theorem \ref{th:jm4})
Take, $r_k = G^{\frac{1}{2}} (x_k-x^*)$. Then from the update formula we get the following:
\begin{align*}
   \|r_{k+1}\|^2  & =  \big \langle \left(I - \omega \alpha_{i_k} G^{-1} Z_{i_k} \right) (x_k-x^*) , \left(I - \omega \alpha_{i_k} G^{-1} Z_{i_k} \right) (x_k-x^*) \big \rangle_G \nonumber \\
    & = \Big \langle \left(I - \omega \alpha_{i_k} G^{-\frac{1}{2}}Z_{i_k} G^{-\frac{1}{2}} \right) r_k , \left(I - \omega \alpha_{i_k} G^{-\frac{1}{2}}Z_{i_k} G^{-\frac{1}{2}} \right) r_k \Big \rangle  = r_k^T \left(I-\omega \alpha_{i_k} T_{i_k}\right)^2  r_k
\end{align*}
Now, using the expression of $\alpha_{i_k}$ we have the following:
\begin{align}
r_k^T \left(I-\omega \alpha_{i_k} T_{i_k}\right)^2 r_k & = \|r_k\|^2 - 2 \omega \alpha_{i_k} r_k^T T_{i_k} r_k + \omega^2 \alpha^2_{i_k} r_k^T T^2_{i_k} r_k \nonumber \\
    & = \|r_k\|^2 - 2 \omega \frac{(r_k^T T_{i_k} r_k) (r_k^T T^2_{i_k} r_k) }{r_k^T T^3_{i_k} r_k}  + \omega^2 \frac{ (r_k^T T^2_{i_k} r_k)^3}{(r_k^T T^3_{i_k} r_k)^2}  \nonumber \\
    & = \|r_k\|^2  - \frac{(r_k^T T_{i_k} r_k) (r_k^T T^2_{i_k} r_k) }{r_k^T T^3_{i_k} r_k}  \left (2 \omega - \omega^2 \frac{(r_k^T T^2_{i_k} r_k)^2}{(r_k^T T^3_{i_k} r_k) (r_k^T T_{i_k} r_k)} \right) \nonumber \\
    & \leq  \|r_k\|^2  - \frac{(r_k^T T_{i_k} r_k) (r_k^T T^2_{i_k} r_k) }{r_k^T T^3_{i_k} r_k}   (2 \omega - \omega^2 ) \label{jm15} \\
    & \leq \|r_k\|^2  -  (2 \omega - \omega^2) \frac{ \|x_k-x^*\|^2_{Z_{i_k}}}{\mu_2} \label{jm16}
\end{align}
here, we used the bound for $\alpha_{i_k}$. Taking expectation in \eqref{jm16} we have the following
\begin{align}
\label{jm17}
    \E \left[\|r_{k+1}\|^2 \ | \ i_k \sim \mathcal{R}\right] 
    & \leq \|r_k\|^2  -  \frac{2(2 \omega - \omega^2)}{\mu_2} f(x_k) \leq (1- \frac{(2 \omega - \omega^2) \lambda^+_1}{\mu_2}) \|r_k\|^2
\end{align}
Taking expectation again and using the tower property, we get the result when $T_{i_k} \succeq 0$. In the second part, we assumed $T_{i_k} \succ 0$, from \eqref{jm15} we have the following:
\begin{align}
\label{jm18}
 \|r_k\|^2 \left[1- (2 \omega - \omega^2) \frac{\|r_k\|^2_{T_{i_k}} \|r_k\|^2_{T^2_{i_k}}}{\|r_k\|^2\|r_k\|^2_{T^3_{i_k}}}  \right]  \leq \|r_k\|^2 \left[1- (2 \omega - \omega^2) \frac{4 \mu_1(i_k) \mu_2(i_k)}{\left[\mu_2(i_k)+ \mu_1(i_k)\right]^2}\right]
\end{align}
here, we used Lemma \ref{lem:strang} with the choice $u = T^{\frac{1}{2}}_{i_k}r_k$ and $v = T_{i_k}r_k$. Now taking expectation with respect to index $i_k$ and considering \eqref{jm15} and \eqref{jm18}, we get
\begin{align}
\label{jm19}
    \E \left[\|r_{k+1}\|^2 \ | \ i_k \sim \mathcal{R}\right] 
    &  \leq  \|r_k\|^2 \left\{1- 4(2 \omega - \omega^2) \E \left[ \frac{ 4\sigma_{i} }{\left(1+ \sigma_{i}\right)^2} \right]\right\}
\end{align}
Taking expectation again in \eqref{jm19} and using the tower property we get the result. Now, to prove the average iterate result, first note that from \eqref{jm17} we have the following:
\begin{align}
\label{jm14}
    \frac{2(2\omega - \omega^2)}{\mu_2} \sum \limits_{l=0}^{k-1} \E[f(x_k)] \leq \sum \limits_{l=0}^{k-1} \left(\E[\|r_k\|^2] - \E[\|r_{k+1}\|^2]\right) \leq \|r_0\|^2 = \|x_0-x^*\|^2_G
\end{align}
Therefore, we have
\begin{align*}
  \E[\|\Tilde{x}_k-x^*\|_G^2]  & =  \E \left[\Big \| \frac{1}{k} \sum \limits_{l=1}^{k} \left(x_l-x^*\right)\Big \|^2_G\right] \leq  \E \left[\frac{1}{k} \sum \limits_{l=1}^{k} \big \| x_l-x^*\big \|^2_G\right] = \frac{1}{k} \sum \limits_{l=1}^{k} \E[\|e_l\|^2] \nonumber \\
  & \leq \frac{2}{k \lambda_1^+} \sum \limits_{l=1}^{k} \E[f(x_l)] \leq \frac{\mu_2 \|x_0-x^*\|^2_G }{\omega k \lambda_1^+(2-\omega)}.
\end{align*}
This proves the average iterate result of Theorem \ref{th:jm4}.
\end{proof}

\begin{proof} (Theorem \ref{th:ssd3})
From the update formula, we have the following:
\begin{align}
\label{jm20}
  \|x_{k+1}-x^*\|_{Z_{i_k}}^2 & = \|x_{k}-x^* - \omega \alpha_{i_k} G^{-1}Z_{i_k}(x_k-x^*) \|_{Z_{i_k}}^2 \nonumber \\
  & = \|x_{k}-x^*\|_{Z_{i_k}}^2 - 2 \alpha_{i_k} r_k^T T_{i_k} r_k + \omega^2 \alpha_{i_k}^2  r_k^T T^3_{i_k} r_k \nonumber \\
  & = \|x_{k}-x^*\|_{Z_{i_k}}^2 - \alpha_{i_k} (2\omega-\omega^2)  r_k^T T^2_{i_k} r_k
\end{align}
Now, the above identity can be simplified as follows:
\begin{align}
\label{jm21}
   \frac{\|x_{k+1}-x^*\|_{Z_{i_k}}^2}{\|x_{k}-x^*\|_{Z_{i_k}}^2}  = 1 -  \frac{(2\omega-\omega^2)(r_k^T T_{i_k}^2 r_k)^2 }{(r_k^TT_{i_k}r_k)(r_k^T T_{i_k}^3 r_k)} \leq 1-  \frac{\mu^+_1(i_k) (2\omega-\omega^2)}{\mu_2(i_k)} \leq 1-  \frac{(2\omega-\omega^2)}{\sigma_{i_k}}
\end{align}
Similarly, considering \eqref{jm20} again and using \eqref{t:7}, we get
\begin{align}
\label{jm22}
 \frac{\|x_{k+1}-x^*\|_{Z_{i_k}}^2}{\|x_{k}-x^*\|_{Z_{i_k}}^2}  = 1 - (2\omega-\omega^2) \frac{(r_k^T T_{i_k}^2 r_k)^2 }{(r_k^TT_{i_k}r_k)(r_k^T T_{i_k}^3 r_k)} \leq 1-  \frac{4 (2\omega-\omega^2)}{4+ \sigma_{i_k}^2}
\end{align}
Now, taking expectation in \eqref{jm21} and \eqref{jm22}, then substituting the results in \eqref{jm20} we get,
 \begin{align*}
\E \left[\frac{f_{i_k}(x_{k+1})}{f_{i_k}(x_{k})}\right]=  \E \left[\frac{\|x_{k+1}-x^*\|_{Z_{i_k}}^2}{\|x_{k}-x^*\|_{Z_{i_k}}^2}\right]  \leq 1-  \frac{4 (2 \omega - \omega^2) }{\min \{4 \E[\sigma_i], 4 + \E[\sigma_i^2]\}}
\end{align*}
this proves the first part of the Theorem. Now, if $T_{i_k} \succ 0$ for each $i_k$, then considering \eqref{jm20} along with \eqref{t:7}, we get the following:
\begin{align}
\label{jm23}
 \frac{\|x_{k+1}-x^*\|_{Z_{i_k}}^2}{\|x_{k}-x^*\|_{Z_{i_k}}^2}  = 1 - (2\omega-\omega^2) \frac{(r_k^T T_{i_k}^2 r_k)^2 }{(r_k^TT_{i_k}r_k)(r_k^T T_{i_k}^3 r_k)} \leq 1- \frac{ 4\sigma_{i_k} (2 \omega - \omega^2)}{\left(1+ \sigma_{i_k}\right)^2}
\end{align}
Taking expectation in \eqref{jm23}, we get the required result. Since, we have the following bound
\begin{align}
\label{jm24}
    \E[f(x_{k+1})] = \frac{1}{2} r_{k+1}^T \E[T] r_{k+1} \leq \frac{\lambda_2}{2} \ \|r_{k+1}\|^2
\end{align}
Now, considering the first part results along with \eqref{jm24}, we get the required bounds of the quantity $\E[f(x_{k+1})]$. Furthermore, considering Jensen's inequality along with \eqref{jm14}, we get
\begin{align*}
  \E[f(\Tilde{x}_k)]   & \leq  \E \left[ \frac{1}{k} \sum \limits_{l=1}^{k} f(x_l)\right] = \frac{1}{k} \sum \limits_{l=1}^{k} \E[f(x_l)] \leq   \frac{\mu_2 \|x_0-x^*\|^2_G }{2\omega k (2-\omega)}.
\end{align*}
This proves the Theorem.
\end{proof}

\begin{proof} (Theorem \ref{th:momentum1})
From the update formula of the proposed momentum algorithm, we have,
\allowdisplaybreaks{
\begin{align*}
  \|r_{k+1}\|^2 =  \|x_{k+1}- x^*\|^2_G & =  \big \| x_k - \omega \alpha_{i_k} \ \nabla^{G} f_{i_k}(x_k) + \gamma (x_k-x_{k-1})- x^* \big \|^2_G \nonumber \\
    & =  \underbrace{\big \|x_k - \omega \alpha_{i_k} \ \nabla^{G} f_{i_k}(x_k) - x^* \big \|_G^2}_{M_1} + \gamma^2 \underbrace{\|x_{k}-x_{k-1}\|_G^2}_{M_2} \nonumber \\
    & + 2 \gamma  \underbrace{\big \langle x_k - x^* \ , \ x_{k}-x_{k-1}\big \rangle_G }_{M_3}- 2 \gamma \omega  \underbrace{ \alpha_{i_k}\big \langle \nabla^{G} f_{i_k}(x_k) \ , \ x_{k}-x_{k-1}\big \rangle_G }_{M_4} \nonumber \\
    & = M_1 + \gamma^2 M_2 + 2 \gamma M_3- 2 \gamma \omega M_4
\end{align*}}
Here, we assume that at $k^{th}$ iteration, the sketching matrix matrix $S_{i_k}$ is chosen. Taking expectation with respect to index $i_k$, we get
\allowdisplaybreaks{
\begin{align}
\label{jm1}
  \E_{i_k}[\|r_{k+1}\|^2]  = \E_{i_k}[M_1 ]  + \gamma^2  \E_{i_k}[M_2]  + 2 \gamma  \E_{i_k}[M_3] - 2 \gamma \omega  \E_{i_k}[M_4 ]
\end{align}}
Now, the first term of \eqref{jm1} can be simplified as follows:
\begin{align}
\label{jm2}
\E_{i_k}[M_1 ]  & =  \E_{i_k} \left[\big \langle \left(I - \omega \alpha_{i_k} G^{-1} Z_{i_k} \right) (x_k-x^*) , \left(I - \omega \alpha_{i_k} G^{-1} Z_{i_k} \right) (x_k-x^*) \big \rangle_G \right] \nonumber \\
    & \leq  \|r_k\|^2 - \frac{2 \omega (2-\omega)}{\mu_2} f(x_k)
\end{align}
We can simplify the second term of \eqref{jm1} as follows:
\allowdisplaybreaks{
\begin{align}
\label{jm3}
\E_{i_k}[M_2 ] =   \|r_k -r_{k-1} \|^2 \leq 2  \|r_k \|^2 + 2 \|r_{k-1} \|^2
\end{align}}
Now, we know the following parallelogram identity
\begin{align}
\label{para}
    2 \langle u, v \rangle_G = \|u\|^2_G + \|v\|^2_G - \|u-v\|^2_G
\end{align}
Take, $u = x_k-x^*$ and $v = x_k-x_{k-1}$. Then using the above identity we get,
\begin{align}
\label{jm4}
  2 \E_{i_k}[M_3 ]  & =  \big \langle x_k - x^* , x_k - x_{k-1} \big \rangle_G  =  \|x_k - x^* \|^2_G + \|x_k-x_{k-1}\|^2_G - \|x_{k-1} - x^* \|^2_G \nonumber \\
    & \leq 3 \|x_k - x^* \|^2_G +  \|x_{k-1} - x^* \|^2_G =  3 \|r_k\|^2 +  \|r_{k-1}\|^2
\end{align}
Now, we need to simplify the fourth term of \eqref{jm1}. Now, we can simplify $M_4$ as follows:
\begin{align*}
    M_4 & = \alpha_{i_k} \big \langle G G^{-1}  Z_{i_k} (x_k-x^*) \ , \ x_{k}-x_{k-1}\big \rangle \nonumber \\
   & = \alpha_{i_k} \big \langle \nabla f_i(x_k) \ , \ x_{k}-x_{k-1}\big \rangle  \geq \alpha_{i_k} f_i(x_k)- \alpha_{i_k} f_i(x_{k-1}) \geq \frac{f_i(x_k)}{\mu_2}- \frac{f_i(x_{k-1})}{\mu^+_1}
\end{align*}
here, we used the convexity of the function $f_i(x)$. Now, the fourth term can be simplified as follows:
\begin{align}
\label{jm5}
   2\E_{i_k}[M_4]   \geq \frac{2 f(x_k)}{\mu_2}- \frac{2 f(x_{k-1})}{\mu^+_1}
\end{align}
 Now, substituting the expressions of \eqref{jm2}, \eqref{jm3}, \eqref{jm4} and \eqref{jm5} in \eqref{jm1}, we get
\allowdisplaybreaks{
\begin{align}
\label{jm6}
   \E_{i_k}[\|r_{k+1}\|^2]  & \leq (1+3 \gamma + 2 \gamma^2) \|r_k\|^2+ (\gamma + 2 \gamma^2) \|r_{k-1}\|^2 + \frac{2\gamma \omega}{\mu^+_1}  f(x_{k-1}) - \frac{2\gamma \omega + 4 \omega - 2 \omega^2}{\mu_2} f(x_k)
\end{align}}
Now, adding $\frac{2 \zeta \omega}{\mu_2} f(x_k)$ in both sides of \eqref{jm6}, we get
\allowdisplaybreaks{
\begin{align}
\label{jm7}
   \E_{i_k}[\|r_{k+1}\|^2 & ]  + \frac{2 \zeta \omega}{\mu_2} f(x_k) \leq (1+3 \gamma + 2 \gamma^2) \|r_k\|^2+ (\gamma + 2 \gamma^2) \|r_{k-1}\|^2 + \frac{2\gamma \omega}{\mu^+_1} f(x_{k-1}) \nonumber \\
   & - \frac{\left(2\gamma \omega + 4 \omega - 2 \omega^2 - 2 \zeta \omega\right)}{\mu_2} f(x_k) \nonumber \\
   & \leq \left(1+3 \gamma + 2 \gamma^2 - \frac{\omega}{\mu_2}(\gamma+2-\omega - \zeta) \lambda^+_1 \right) \|r_k\|^2+ \left(\gamma + 2 \gamma^2\right) \|r_{k-1}\|^2 \nonumber  \\
   & + \frac{2\xi \omega}{\mu^+_1} f(x_{k-1}) + \frac{\left(2\gamma \omega  - 2\omega \xi\right)}{\mu^+_1} f(x_{k-1}) \nonumber  \\
   & \leq \left(1+3 \gamma + 2 \gamma^2 - \frac{\omega}{\mu_2}(\gamma+2-\omega - \zeta) \lambda^+_1 \right) \|r_k\|^2 + \frac{2\xi \omega}{\mu^+_1} f(x_{k-1}) \nonumber  \\
   & + \left(\gamma + 2 \gamma^2 + \frac{\omega (\gamma- \xi)}{\mu^+_1} \lambda_2 \right) \|r_{k-1}\|^2 \nonumber  \\
   & = \phi_1 \|r_k\|^2 + \phi_2 \|r_{k-1}\|^2 + \frac{2\xi \omega}{\mu^+_1} f(x_{k-1}) 
\end{align}}
Here, we used the fact $\gamma \geq \max \left\{\xi, \zeta-2+\omega\right\}$. Taking expectation again in \eqref{jm7} and using the tower property, we get the following:
\allowdisplaybreaks{
\begin{align}
\label{jm8}
   \E[\|r_{k+1}\|^2]  + \frac{2 \zeta \omega}{\mu_2} \E[f(x_k)] \leq  \phi_1 \E[\|r_k\|^2] + \phi_2 \E[\|r_{k-1}\|^2] + \frac{2\xi \omega}{\mu^+_1} \E[f(x_{k-1})] 
\end{align}}
Now, the parameter $\delta$ satisfies the following relations:
\begin{align}
    \label{l1}
    \xi \leq \frac{\zeta \mu^+_1 (\delta+\phi_1)}{\mu_2}, \quad \text{and} \quad  \phi_2 \leq \delta (\delta+\phi_1)
\end{align}
Considering equations \eqref{jm8} and \eqref{l1}, we have the following:
\allowdisplaybreaks{
\begin{align}
\label{l2}
 \E\left[ \mathcal{V}_{k+1}\right] & =  \E[\|r_{k+1}\|^2]  + \delta \E[\|r_k\|^2]+ \frac{2 \zeta \omega}{\mu_2} \E[f(x_k)] \nonumber \\
 & \leq  (\phi_1+ \delta) \E[\|r_k\|^2] + \phi_2 \E[\|r_{k-1}\|^2] + \frac{2\xi \omega}{\mu^+_1} \E[f(x_{k-1})]  \nonumber \\
 & \leq  (\phi_1+ \delta) \E[\|r_k\|^2] + \delta (\delta+\phi_1) \E[\|r_{k-1}\|^2] +  \frac{2 \omega \zeta (\delta+\phi_1)}{\mu_2} \E[f(x_{k-1})] \nonumber \\
 & =  (\phi_1+ \delta) \left[ \E[\|r_k\|^2]  + \delta \E[\|r_{k-1}\|^2]+ \frac{2 \zeta \omega}{\mu_2} \E[f(x_{k-1})]\right] = \rho \E\left[ \mathcal{V}_{k}\right] 
\end{align}}
This proves the result. Finally, we have to show that $\rho < 1$. Using the assumption $\phi_1 + \phi_2 < 1$, we have
\begin{align*}
    \rho = \max \left\{\frac{\xi \mu_2}{\zeta \mu^+_1}, \frac{\phi_1+ \sqrt{\phi_1^2+4 \phi_2}}{2} \right\} < \max \left\{1, \frac{\phi_1+ \sqrt{\phi_1^2+4 (1-\phi_1)}}{2} \right\} = 1
\end{align*}
This proves Theorem \ref{th:momentum1}.
\end{proof}

\begin{proof} (Theorem \ref{th:cesaro})
First, let us  define $\vartheta_l = \frac{\gamma}{1-\gamma}[x_{l}-x_{l-1}], \quad  \chi_l = \|x_l+\vartheta_l-x^*\|_G^2$ for any natural number $l \geq 1$. Also assume that the sketching matrix $S_{i_l}$ is chosen at iteration $l$. Now, using the update formula of \eqref{jml16}, we get
\begin{align}
\label{ces:1}
    \chi_{l+1} & = \|x_{l+1}+\vartheta_{l+1}-x^*\|_G^2   = \big \| x_l + \vartheta_l - \frac{\omega \alpha_{i_l}}{1-\gamma} \nabla^{G} f_{i_l}(x_l)-x^* \big \|_G^2 \nonumber \\
    & =  \|x_l+\vartheta_l-x^*\|_G^2 + \frac{\omega^2 \alpha^2_{i_l}}{(1-\gamma)^2} \|\nabla^{G} f_{i_l}(x_l)\|_G^2  - \frac{2 \omega \alpha_{i_l}}{1-\gamma}   \big \langle x_l+\vartheta_l-x^* \ ,\  \nabla^{G} f_{i_l}(x_l) \big \rangle_G \nonumber \\
    & = \chi_l + \frac{2 \omega^2 }{(\mu^+_1)^2 (1-\gamma)^2} f_{i_l}(x_l) - \frac{2 \omega }{1-\gamma}  \underbrace{\alpha_{i_l} \big \langle x_l+\vartheta_l-x^* \ ,\  \nabla^{G} f_{i_l}(x_l) \big \rangle_G}_{J} .
\end{align}
The third term can be simplified as follows:
\begin{align}
  \label{ces:3}
   - \frac{2\omega}{1-\gamma} J & = -\frac{2\omega}{1-\gamma} \big \langle x_l-x^* , \alpha_{i_l} Z_{i_l}(x_l-x^*) \big \rangle + \frac{2\omega \gamma}{(1-\gamma)^2} \big \langle x_{l-1}-x_l , \alpha_{i_l} Z_{i_l}(x_l-x^*) \big \rangle \nonumber \\
  & \leq  - \frac{4\omega}{1-\gamma} \alpha_{i_l} f_i(x_l) +  \frac{2\omega \gamma}{(1-\gamma)^2} \alpha_{i_l} \left[f_i(x_{l-1}) -  f_i(x_{l}) \right] \nonumber \\
  & \leq  - \frac{4\omega}{\mu_2 (1-\gamma)}  f_i(x_l) +  \frac{2\omega \gamma}{\mu^+_1 (1-\gamma)^2} f_i(x_{l-1}) - \frac{2\omega \gamma}{\mu_2 (1-\gamma)^2}  f_i(x_{l}) .
\end{align}
Taking expectation on both sides of \ref{ces:1} with respect to index $i_l$, we get
\begin{align}
\label{ces:2}
   \E_{i_l}[ \chi_{l+1}] & \leq \chi_l + \frac{2 \omega^2 f(x_l) }{\mu_2^2 (1-\gamma)^2} - \frac{2 \omega \E_{i_l}[J]}{1-\gamma}    \leq \chi_l + \frac{2\omega \gamma f(x_{l-1})}{\mu^+_1 (1-\gamma)^2}  - \frac{2\omega \gamma  f(x_{l})}{\mu^+_1 (1-\gamma)^2} - \varpi f(x_l)
\end{align}
where, the constant $\varpi$ is taken as
\begin{align*}
   \varpi & =   \frac{4\omega}{\mu_2 (1-\gamma)}  -  \frac{2\omega^2}{\mu_2^2(1-\gamma)^2}+ \frac{2\omega \gamma}{\mu_2 (1-\gamma)^2}- \frac{2\omega \gamma}{\mu^+_1 (1-\gamma)^2} =  \frac{2 \omega \left[2- \frac{\omega}{\mu_2}- \gamma \left(1+\frac{\mu_2}{\mu^+_1}\right) \right]}{\mu_2 (1-\gamma)^2}  >  0.
\end{align*}
Simplifying equation \eqref{ces:2} further, we get
\begin{align}
    \label{ces:4}
   \E_{i_l}[ \chi_{l+1} ] + \frac{2\omega \gamma}{\mu^+_1 (1-\gamma)^2} f(x_l) + \varpi f(x_l) \ \leq \  \chi_{l} + \frac{2\omega \gamma}{\mu^+_1 (1-\gamma)^2} f(x_{l-1}),
\end{align}
Taking expectation again in \eqref{ces:4} and using the tower property of expectation, we get the following recurrence
\begin{align}
\label{ces:6}
    y_{l+1} + \varpi \E[f(x_l)] \leq y_l, \quad l = 1,2,3...
\end{align}
with the definition: $y_l =\E[\chi_{l}] + \frac{2\omega \gamma}{\mu^+_1 (1-\gamma)^2} \E [f(x_{l-1})] $. Summing up \eqref{ces:6} 
for $l \in [1,k]$, we have
\begin{align}
    \label{ces:7}
    \sum \limits_{l=1}^{k} \E [f(x_l)] \ \leq \ \frac{y_1-y_{k+1}}{\varpi} \ \leq \ \frac{y_1}{\varpi}
\end{align}
Furthermore, using the convexity property of function $f$, we get
\begin{align}
\label{ces:8}
    \E \left[f(\bar{x_k})\right] = \E \left[f\left(\sum \limits_{l=1}^{k} \frac{x_l}{k}\right)\right] \ \leq \ \E \left[\frac{1}{k} \sum \limits_{l=1}^{k}f(x_l)\right] \ = \ \frac{1}{k}  \sum \limits_{l=1}^{k} \E [f(x_l)] \ \overset{\eqref{ces:7}}{\leq}  \frac{y_1}{\varpi k}
\end{align}
From construction, we have $x_0=x_1$. Now, it can easily check that, $\vartheta_1 = 0$ and $\chi_1 = \|x_0 -x^*\|_G^2$. Therefore, from the definition of sequence $y_1$, we get
\begin{align*}
   y_1  =\E[\chi_{1}] + \frac{2\omega \gamma}{\mu^+_1 (1-\gamma)^2} \E [f(x_{0})]  = \ \|x_0 -x^*\|_G^2  + \frac{2\omega \gamma}{\mu^+_1 (1-\gamma)^2} f(x_0).
\end{align*}
Finally, substituting the values of $y_1$ and $\varpi$ in \eqref{ces:8}, we have the following
\begin{align*}
    \E \left[f(\bar{x}_k)\right] \leq \frac{ \mu^+_1 \mu_2 (1-\gamma)^2 \ \|x_0 -x^*\|_G^2 + 2 \gamma \omega \mu_2 f(x_0)}{2 \omega k \left(2\mu^+_1 \mu_2 - \gamma \mu^+_1 \mu_2 - \gamma \mu_2^2 -\omega \mu^+_1 \right)}.
\end{align*}
This proves the Theorem.
\end{proof}



\section{Additional Experiments}
\label{appendix:exp}

\begin{figure}[htbp]
  \vspace{- 10 pt}
    \includegraphics[scale = 0.35]{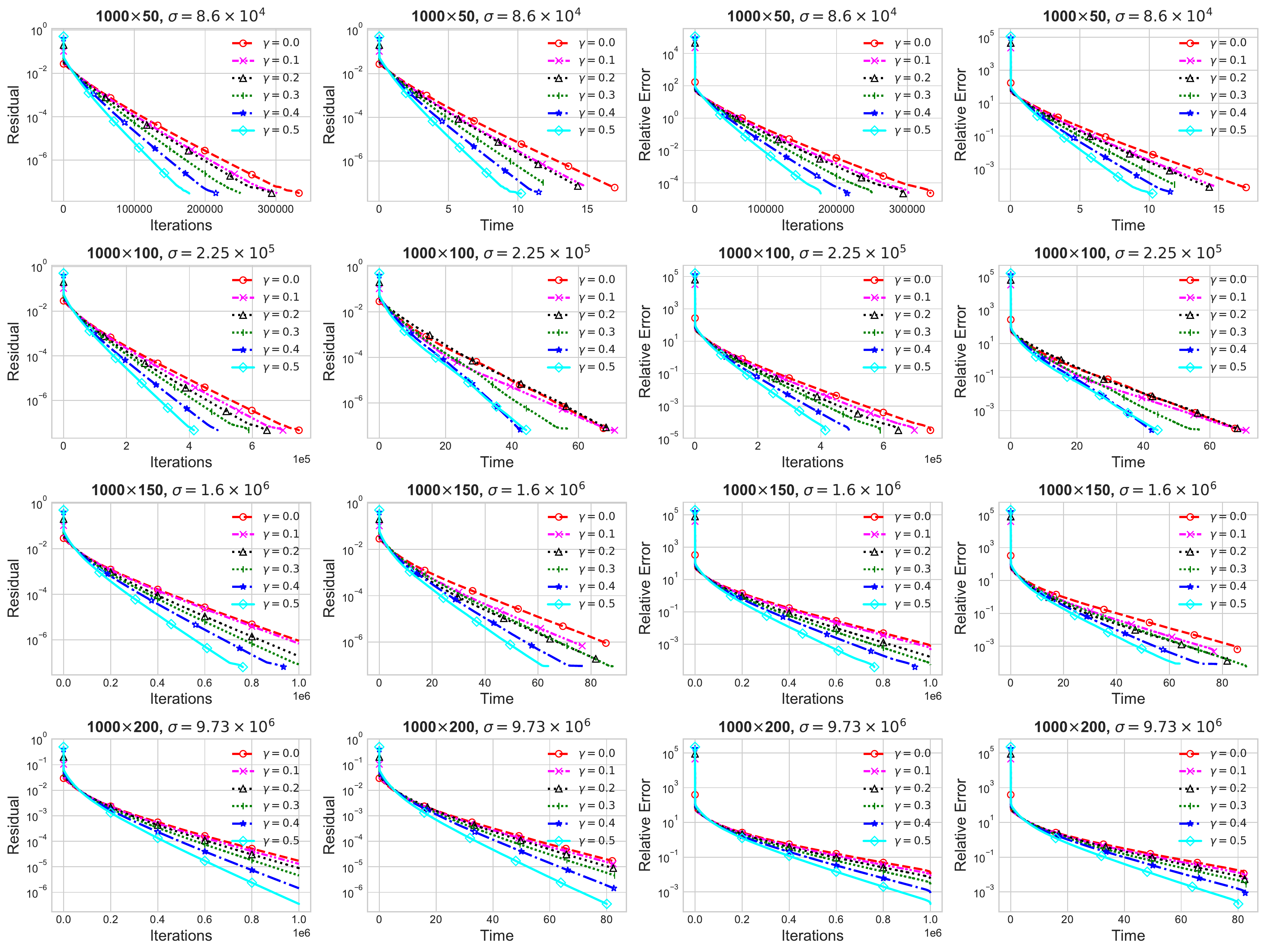}
    \caption{GK with momentum (sketch Sample size, $\tau = 5$): comparison among momentum variants on Gaussian data, residual error and relative error vs CPU time and No. of iterations.}
    \label{fig:55}
      \vspace{- 10 pt}
\end{figure}

\newpage
\begin{figure}[ht!]
    \includegraphics[scale = 0.37]{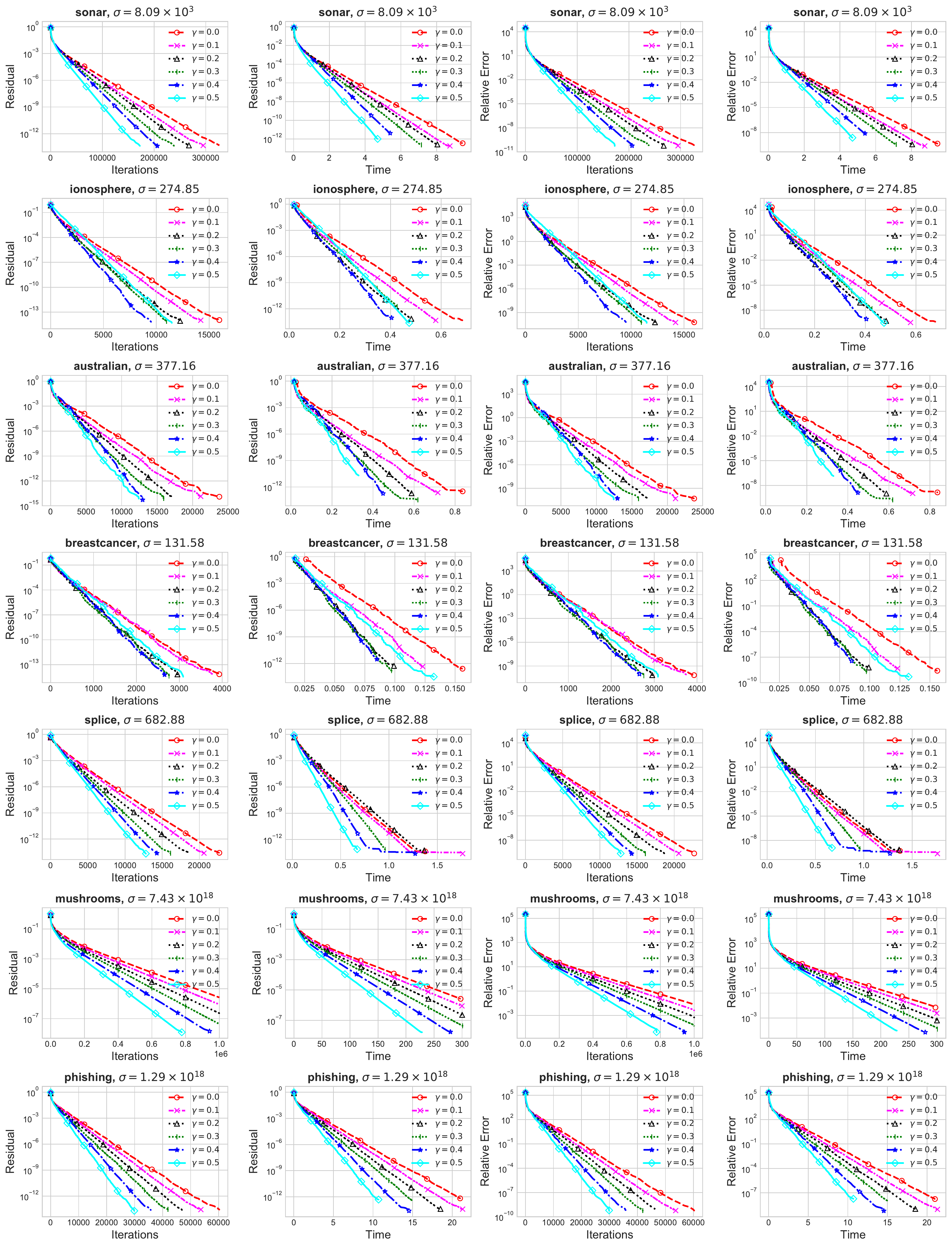}
    \caption{GK with momentum (sketch Sample size, $\tau = 1$): comparison among momentum variants on LIBSVM data, residual error and relative error vs CPU time and No. of iterations.}
    \label{fig:7}
\end{figure}

\newpage

\begin{figure}[ht!]
    \includegraphics[scale = 0.37]{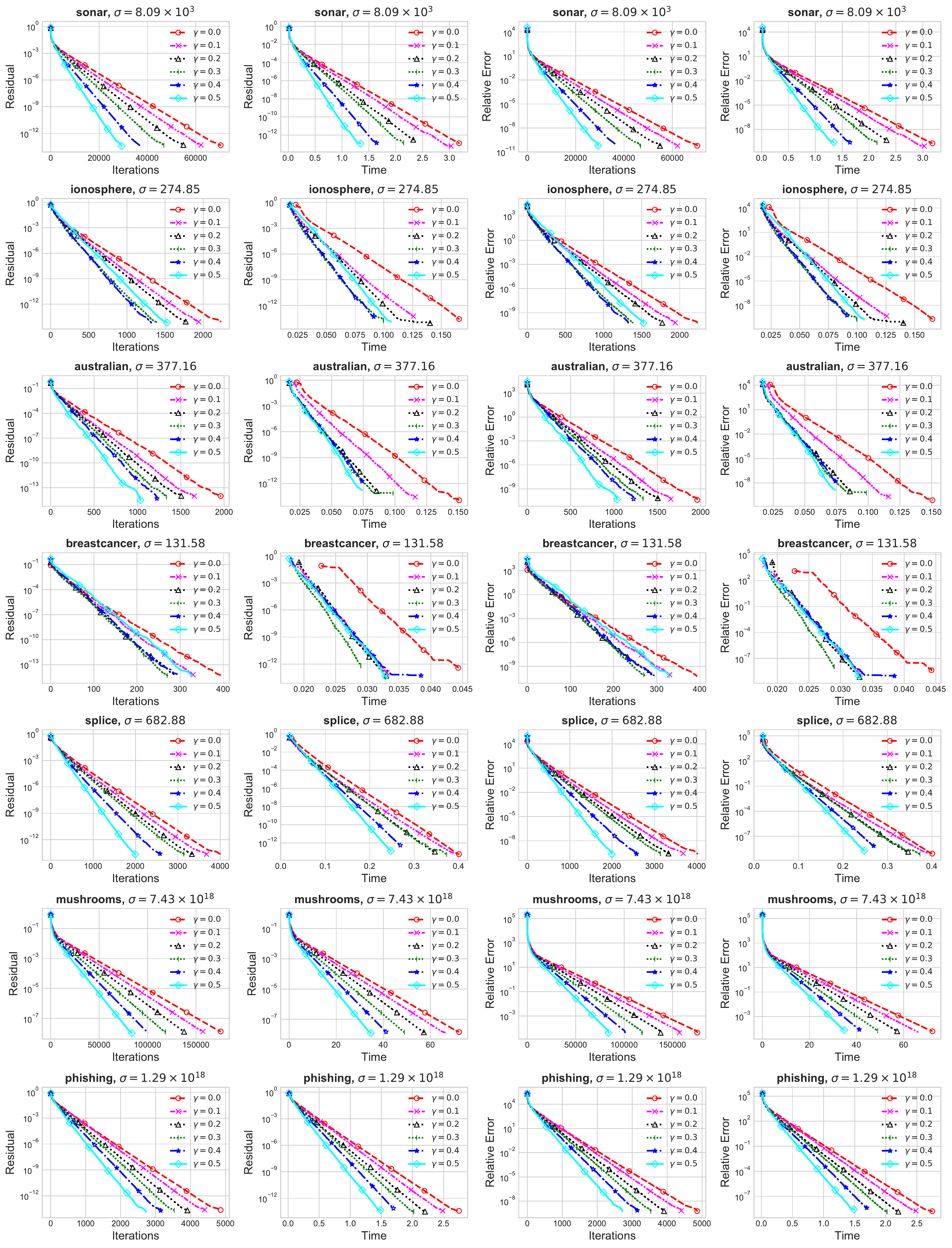}
    \caption{GK with momentum (sketch Sample size, $\tau = 20$): comparison among momentum variants on LIBSVM data,  residual error and relative error vs CPU time and No. of iterations.}
    \label{fig:8}
\end{figure}

\newpage

\begin{figure}[ht!]
    \includegraphics[scale = 0.37]{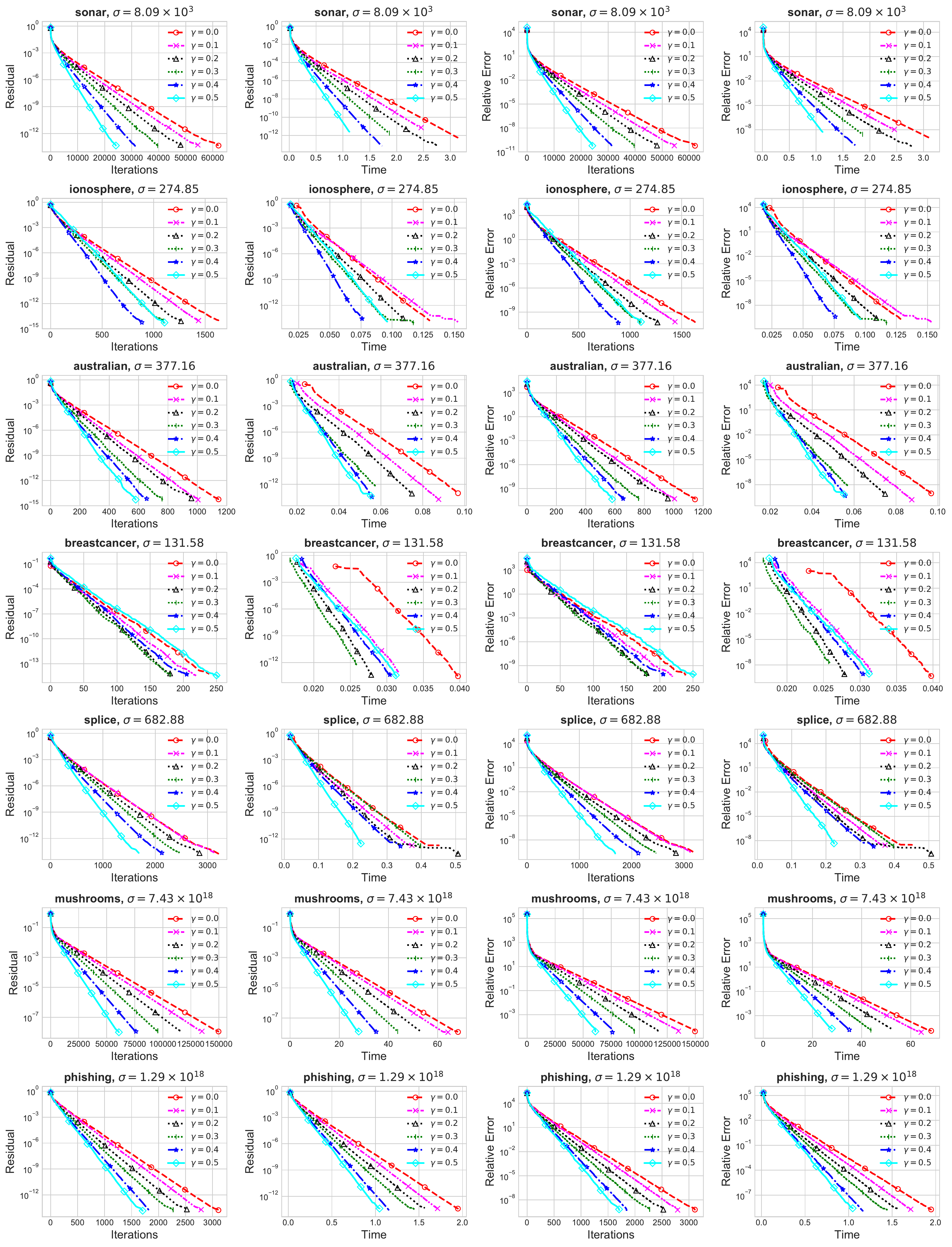}
    \caption{GK with momentum (sketch Sample size, $\tau = 50$): comparison among momentum variants on LIBSVM data, residual error and relative error vs CPU time and No. of iterations.}
    \label{fig:9}
\end{figure}

\newpage

\begin{figure}[ht!]
    \includegraphics[scale = 0.37]{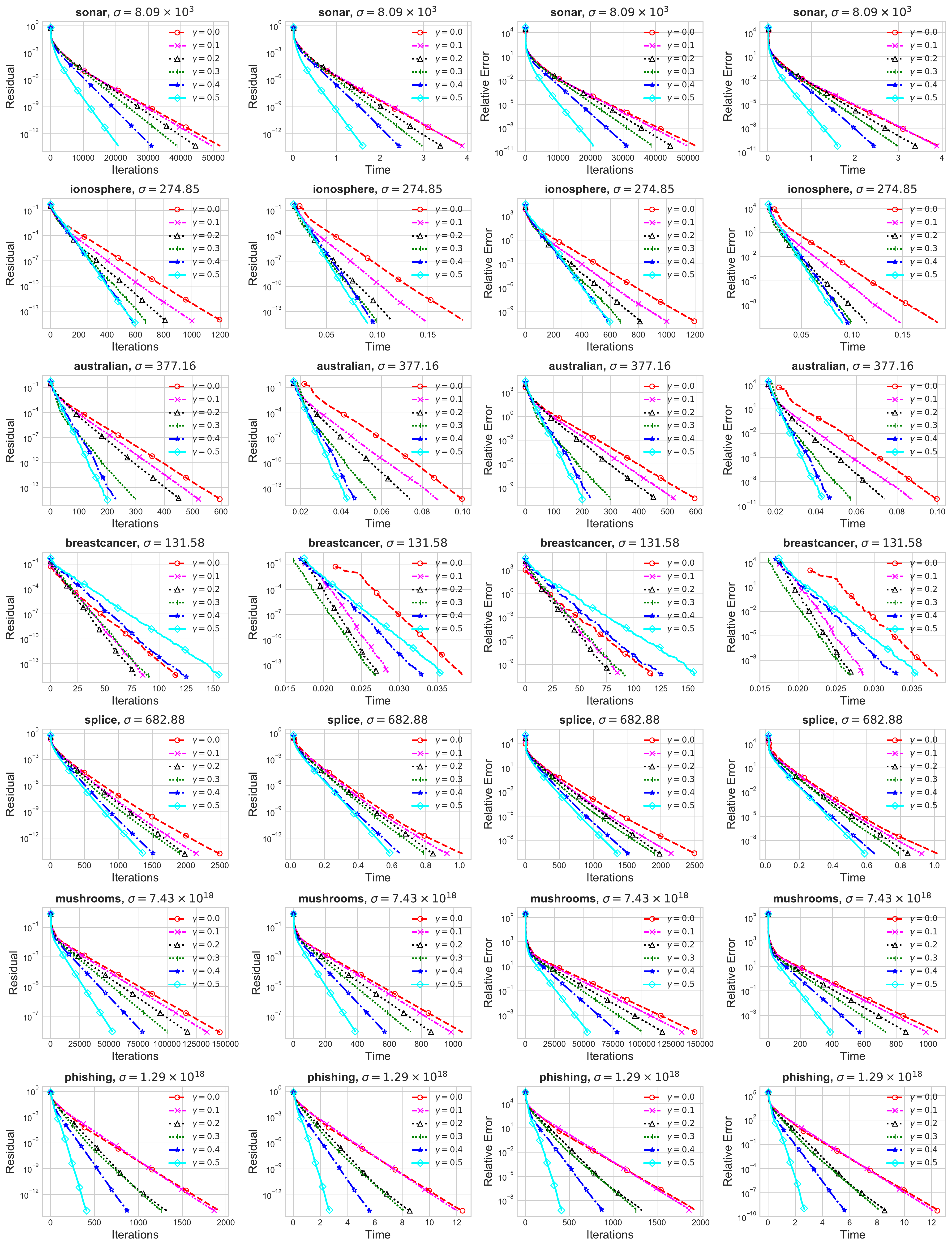}
    \caption{GK with momentum (sketch Sample size, $\tau = m$): comparison among momentum variants on LIBSVM data, residual error and relative error vs CPU time and No. of iterations.}
    \label{fig:10}
\end{figure}

\newpage

\begin{figure}[ht!]
    \includegraphics[scale = 0.37]{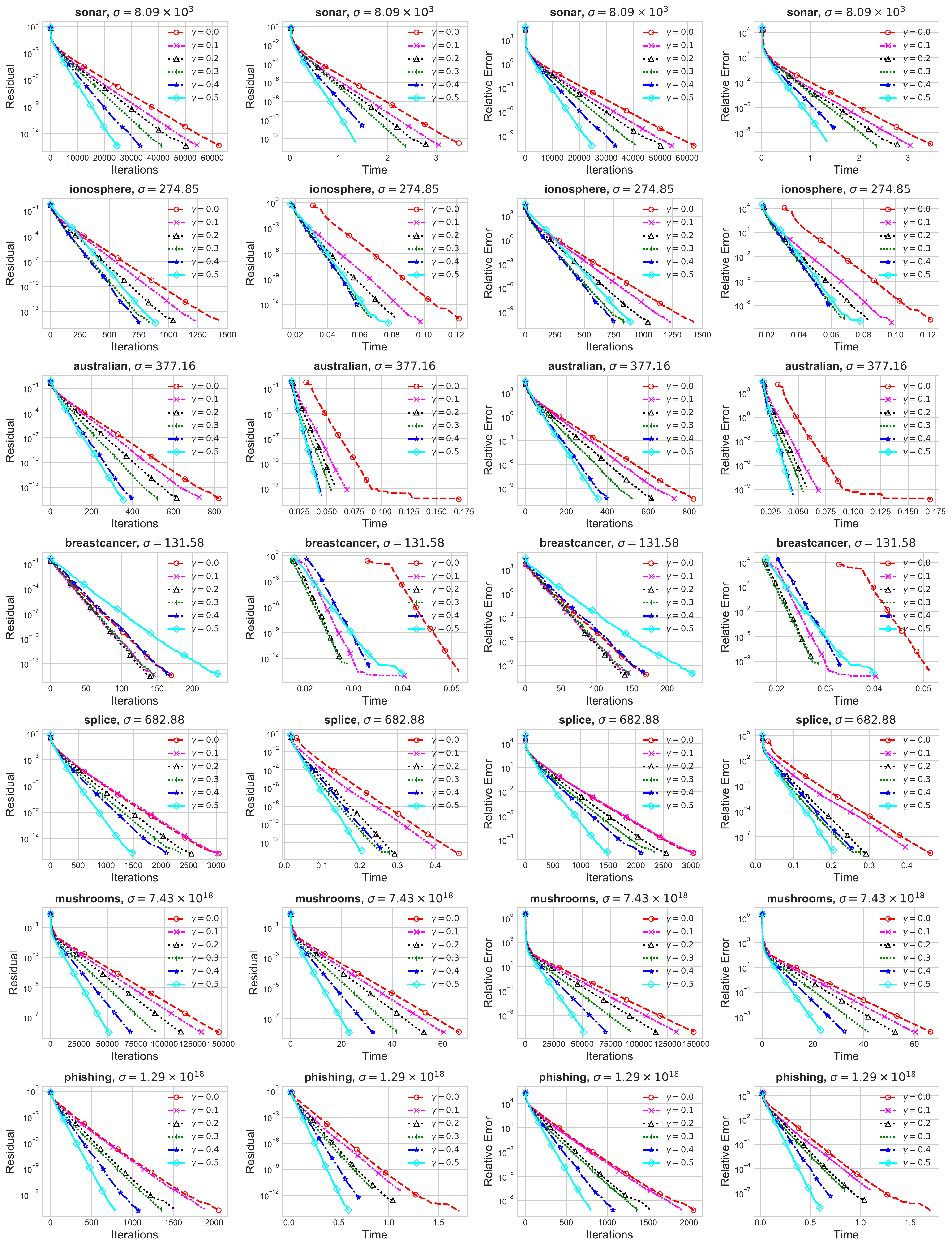}
    \caption{GK with momentum (capped, $\theta = 0.5, \tau_1 = 1, \tau_2 = m$): comparison among momentum variants on LIBSVM data, residual error and relative error vs CPU time and No. of iterations.}
    \label{fig:11}
\end{figure}

\newpage

\begin{figure}[ht!]
    \includegraphics[scale = 0.37]{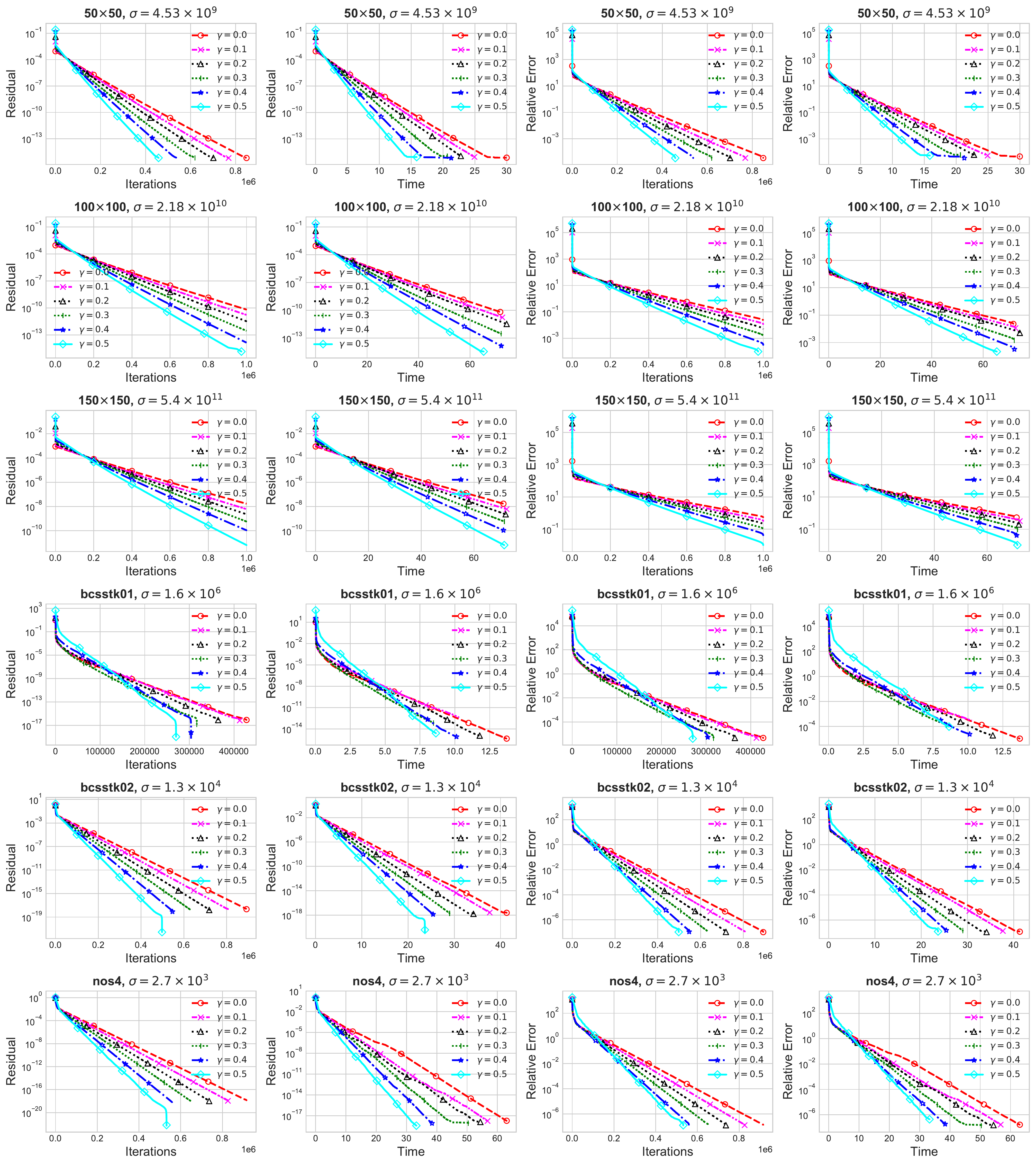}
    \caption{GCD with momentum (sketch Sample size, $\tau = 1$): comparison among momentum variants on Gaussian and Matrix market data, residual error and relative error vs CPU time and No. of iterations.}
    \label{fig:12}
\end{figure}

\newpage

\begin{figure}[htbp]
    \includegraphics[scale = 0.37]{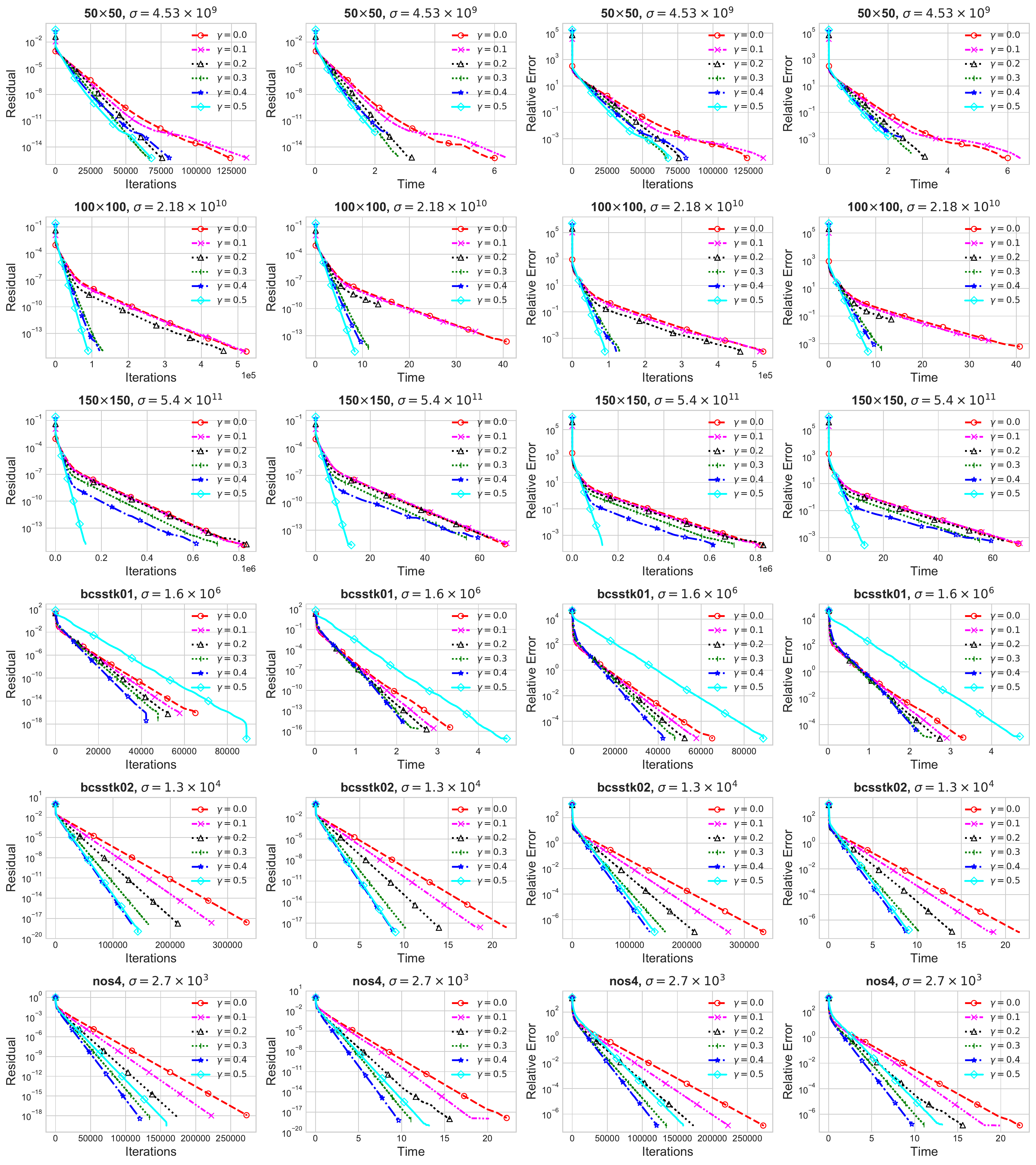}
    \caption{GCD with momentum (sketch Sample size, $\tau = 10$): comparison among momentum variants on Gaussian and Matrix market data, residual error and relative error vs CPU time and No. of iterations.}
    \label{fig:13}
\end{figure}

\newpage

\begin{figure}[ht!]
    \includegraphics[scale = 0.37]{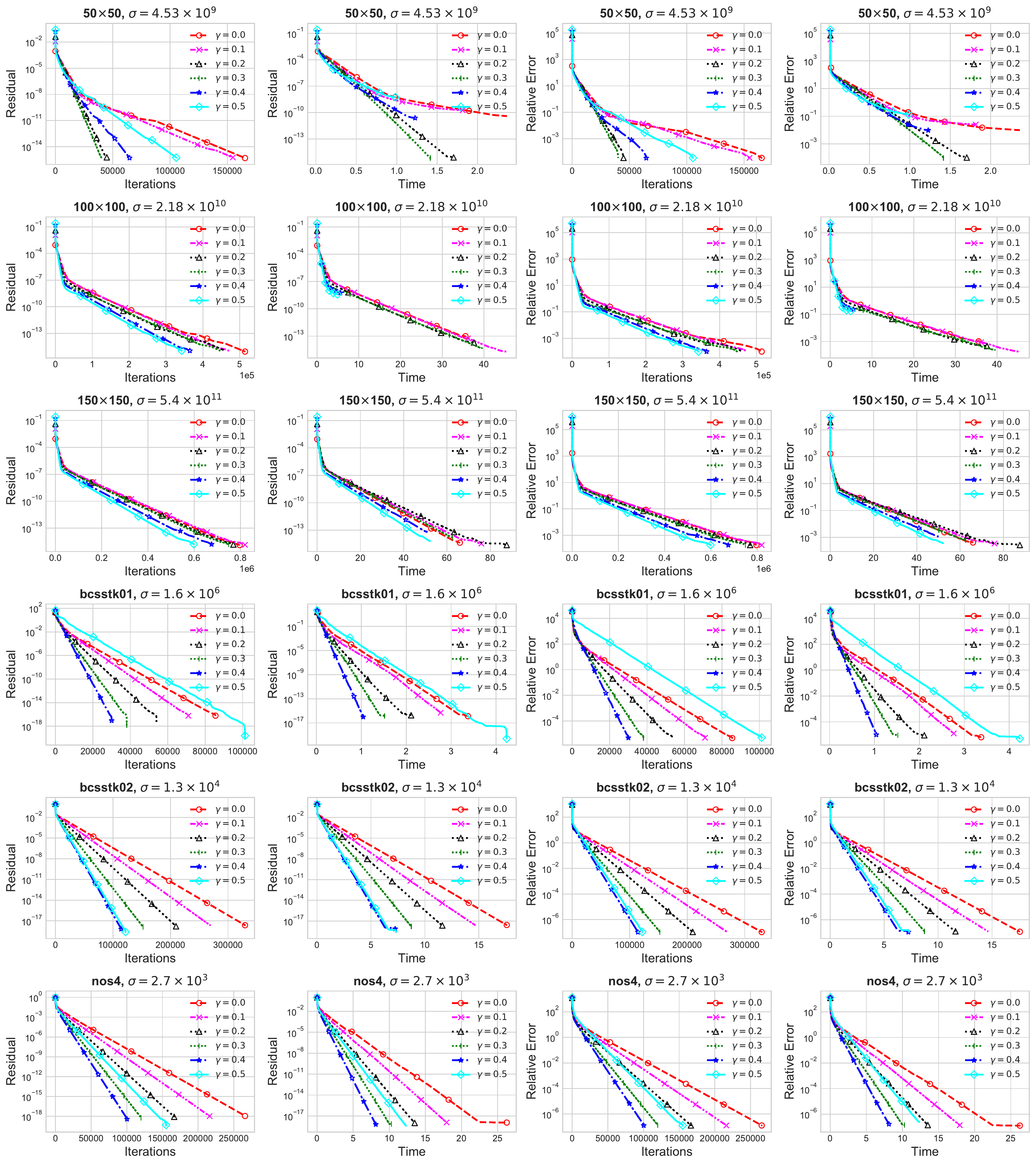}
    \caption{GCD with momentum (sketch Sample size, $\tau = 20$): comparison among momentum variants on Gaussian and Matrix market data, residual error and relative error vs CPU time and No. of iterations.}
    \label{fig:14}
\end{figure}

\newpage

\begin{figure}[ht!]
    \includegraphics[scale = 0.37]{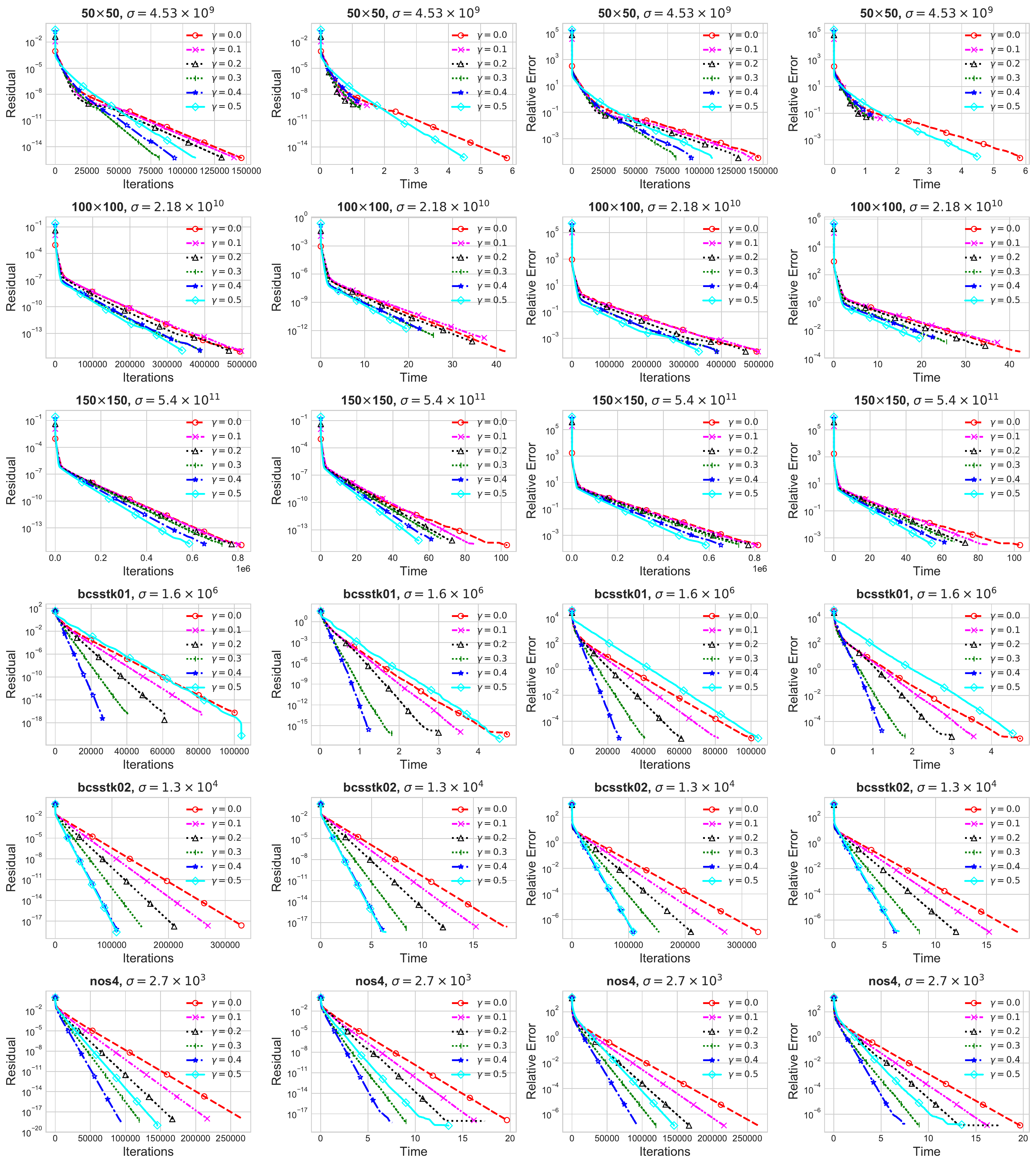}
    \caption{GCD with momentum (sketch Sample size, $\tau = 30$): comparison among momentum variants on Gaussian and Matrix market data, residual error and relative error vs CPU time and No. of iterations.}
    \label{fig:15}
\end{figure}

\newpage
\begin{figure}[ht!]
    \includegraphics[scale = 0.37]{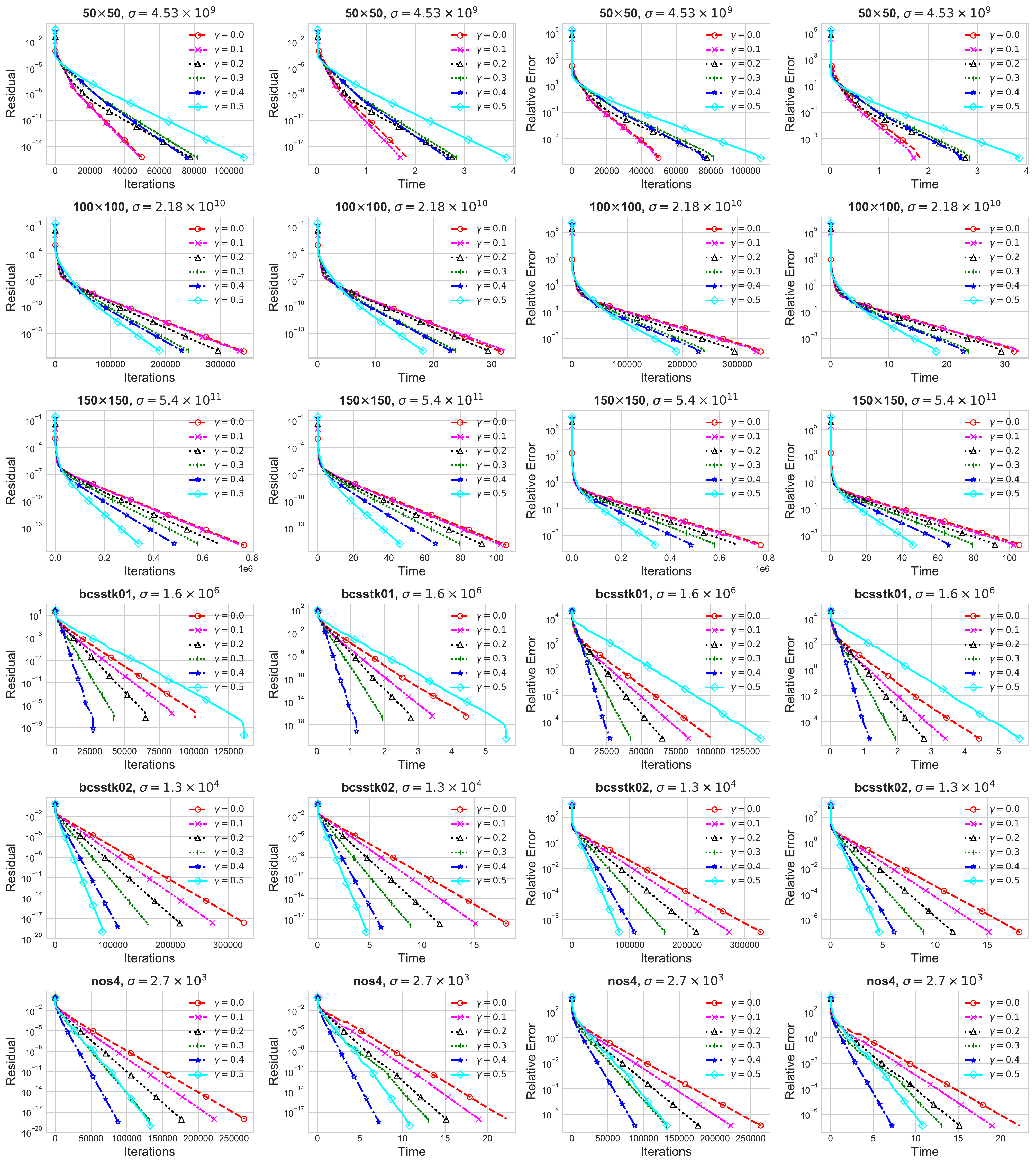}
    \caption{GCD with momentum (sketch Sample size, $\tau = m$): comparison among momentum variants on Gaussian and Matrix market data, residual error and relative error vs CPU time and No. of iterations.}
    \label{fig:16}
\end{figure}

\newpage

\begin{figure}[ht!]
    \includegraphics[scale = 0.37]{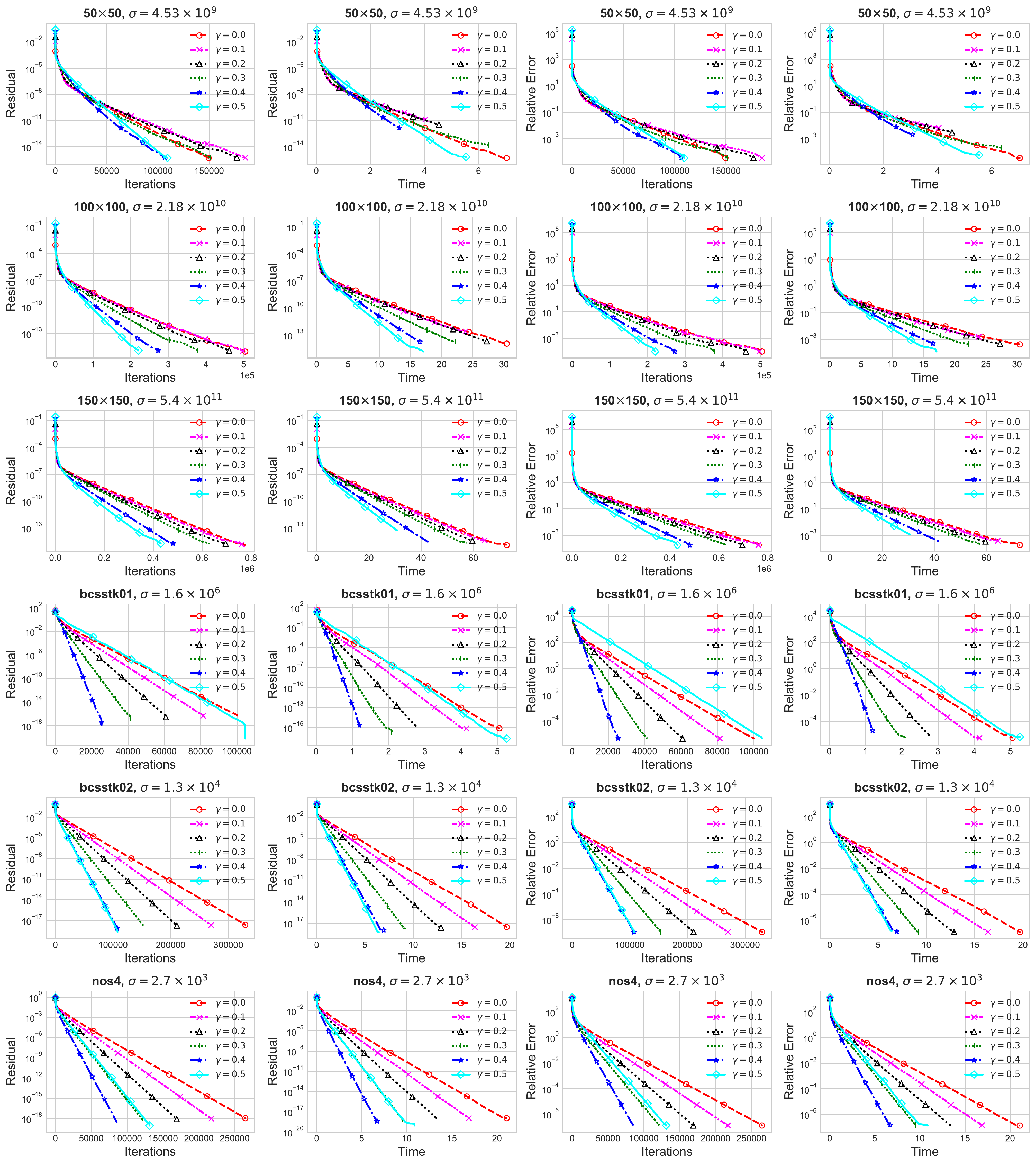}
    \caption{GCD with momentum (capped, $\theta = 0.5, \tau_1 = 1, \tau_2 = m$): comparison among momentum variants on Gaussian and Matrix market data, residual error and relative error vs CPU time and No. of iterations.}
    \label{fig:17}
\end{figure}

\newpage

\bibliography{sample}

\end{document}